\documentclass[11pt]{amsart}
\usepackage[utf8]{inputenc}

\usepackage{enumitem}
\usepackage{amsfonts}
\usepackage{amsmath}
\usepackage{amsthm}
\usepackage{amssymb}
\usepackage{mathrsfs}
\usepackage{enumitem}
\usepackage{forest}
 
\usepackage[all]{xy}
\usepackage{appendix}

\usepackage{tikz}
\usetikzlibrary{shapes.geometric}
\usetikzlibrary{positioning}
\usetikzlibrary{automata}

\usepackage{tikz-cd}

\theoremstyle{definition}
\newtheorem{defi}{Definition}[section]

\newtheorem{rem}[defi]{Remark}
\newtheorem{prop}[defi]{Proposition}

\newtheorem{theo}[defi]{Theorem}

\newtheorem{con}{Construction}

\newcommand{\C}{\mathsf{C}}

\newcommand{\dgvect}{\mathsf{DGVect}(\mathsf{k})}

\newcommand{\id}{\operatorname{id}}
\newcommand{\aaa}{\mathbf{a}}
\newcommand{\m}{\mathbf{m}}

\newcommand{\co}{\colon\thinspace}

\makeatletter
\newcommand\RSloop{\@ifnextchar\bgroup\RSloopa\RSloopb}
\makeatother
\newcommand\RSloopa[1]{\bgroup\RSloop#1\relax\egroup\RSloop}
\newcommand\RSloopb[1]%
  {\ifx\relax#1%
   \else
     \ifcsname RS:#1\endcsname
       \csname RS:#1\endcsname
     \else
       \GenericError{(RS)}{RS Error: operator #1 undefined}{}{}%
     \fi
   \expandafter\RSloop
   \fi
  }
\newcommand\X{0}
\newcommand\RS[1]%
  {\begin{tikzpicture}
     [every node/.style=
       {circle,draw,fill,minimum size=1.5pt,inner sep=0pt,outer sep=0pt},
      line cap=round
     ]
   \coordinate(\X) at (0,0);
   \RSloop{#1}\relax
   \end{tikzpicture}
  }
\newcommand\RSdef[1]{\expandafter\def\csname RS:#1\endcsname}
\newlength\RSu
\RSdef{i}{\draw (\X) -- +(90:\RSu) node{};}
\RSdef{l}{\draw (\X) -- +(135:\RSu) node{};}
\RSdef{m}{\draw (\X) -- +(120:\RSu) node{};}
\RSdef{q}{\draw (\X) -- +(107:\RSu) node{};}

\RSdef{n}{\draw (\X) -- +(105:\RSu) node{};}
\RSdef{r}{\draw (\X) -- +(45:\RSu) node{};}
\RSdef{s}{\draw (\X) -- +(60:\RSu) node{};}
\RSdef{z}{\draw (\X) -- +(73:\RSu) node{};}

\RSdef{t}{\draw (\X) -- +(75:\RSu) node{};}

\RSdef{I}{\draw (\X) -- +(90:\RSu) coordinate(\X I);\edef\X{\X I}}
\RSdef{B}{\draw[red, very thick] (\X) -- +(90:\RSu) coordinate(\X I);\edef\X{\X I}}

\RSdef{L}{\draw (\X) -- +(135:\RSu) coordinate(\X L);\edef\X{\X L}}
\RSdef{A}{\draw[red, very thick] (\X) -- +(135:\RSu) coordinate(\X L);\edef\X{\X L}}

\RSdef{M}{\draw (\X) -- +(120:\RSu) coordinate(\X L);\edef\X{\X L}}
\RSdef{Q}{\draw (\X) -- +(107:\RSu) coordinate(\X L);\edef\X{\X L}}

\RSdef{N}{\draw (\X) -- +(105:\RSu) coordinate(\X L);\edef\X{\X L}}
\RSdef{R}{\draw (\X) -- +(45:\RSu) coordinate(\X R);\edef\X{\X R}}
\RSdef{C}{\draw[red, very thick] (\X) -- +(45:\RSu) coordinate(\X R);\edef\X{\X R}}

\RSdef{S}{\draw (\X) -- +(60:\RSu) coordinate(\X R);\edef\X{\X R}}
\RSdef{Z}{\draw (\X) -- +(73:\RSu) coordinate(\X R);\edef\X{\X R}}

\RSdef{T}{\draw (\X) -- +(75:\RSu) coordinate(\X R);\edef\X{\X R}}

\newsavebox\sba\savebox\sba[5mm][l]{\RS{{Mm}{Ii}{R{ir}}}}
\newsavebox\sbb\savebox\sbb[5mm][l]{\RS{{Ss}{Ii}{L{il}}}}
\newsavebox\sbc\savebox\sbc[5mm][l]{\RS{{L{lir}}{Rr}}}
\newsavebox\sbd\savebox\sbd[5mm][l]{\RS{{Ll}{Ilr}{Rr}}}
\newsavebox\sbe\savebox\sbe[5mm][l]{\RS{{Ll}{R{lir}}}}

\newsavebox\lbo\savebox\lbo[5mm][l]{\RS{ B {AAAl} {BB{Al}{Bi}{Cr}} {CCCr}}}
\newsavebox\mbo\savebox\mbo[5mm][l]{\RS{BI {Ll} {Ii} {Rr} }}
\newsavebox\rbo\savebox\rbo[5mm][l]{\RS{ B{ {AAlr} {CRlr} }}}

\title{Cellular chains on freehedra and operadic pairs}

\author{Daria Poliakova}

\begin{document}

\maketitle
 \begin{abstract}
The paper is devoted to explaining the operadic meaning of freehedra, a family of polytopes originally defined to study free loop spaces. We introduce the notion of operadic pairs and algebras over them. Cellular chains on Stasheff associahedra and Stasheff multiplihedra assemble into a two-colored operadic pair that governs $A_\infty$-algebras and $A_\infty$-modules over them, with maps that are $A_\infty$ in both colors. An important quotient of this operadic pair is given by cellular chains on cubes and freehedra.
\end{abstract}
 \tableofcontents
\section{Introduction}
 The present paper grew out of the author's attempts to understand and extend the constructions of \cite{ACD}. \\
 
 In \cite{AC}, representations up to homotopy of a derived algebraic group $G$ were introduced as $A_\infty$-comodules over the group coalgebra $A: = \mathcal{O}(G)$. They form a DG-category $\operatorname{Rep}^h(G)$. In \cite{ACD} it was proved that the homotopy category of $\operatorname{Rep}^h(G)$ is monoidal. \\
 
 To construct tensor products, the authors used the language of DB-algebras and DB-bimodules (see section 3.2 and 6.1 in \cite{ACD}) and studied the algebra of $\operatorname{Rep}^h(G)$ by means of a certain universal DB-pair $(\Omega,T)$. The tensor product of objects in $\operatorname{Rep}^h(G)$ was given by a diagonal $\Omega \to \Omega \boxtimes \Omega$, and the tensor product of morphisms was given by a diagonal $T \to T \boxtimes T$. The resulting tensor product of morphisms was only homotopy associative and homotopy consistent with compositions. It was left as an open question whether this monoidal structure admits some sort of a coherent lift to DG-level. \\

 In operadic language, the DB-algebra $\Omega$ corresponds to an $(\aaa,\m)$-colored operad $\Omega$ governing pairs of a DG-algebra and an $A_\infty$-module over it. The DB-bimodule $T$ corresponds to an operadic $\Omega$-bimodule $T$ governing maps of such pairs which are homomorphisms in color $\aaa$ and $A_\infty$ in color $\m$. We axiomatize the situation by defining {\em operadic pairs} and algebras over them. The pair $(\Omega,T)$ provides an example of an operadic pair. \\
 
The context for operadic pairs is as follows. The category of $A_\infty$-algebras is {\em not} the category of algebras over the DG-operad $A_\infty$, because the latter category does not have enough morphisms. However, there exists an operadic pair $(A_\infty, M_\infty)$, for which there is an equivalence of categories $A_\infty\operatorname{Alg} \simeq \operatorname{Alg}(A_\infty, M_\infty)$.  The operadic pair $(A_\infty, M_\infty)$ consists of cellular chains on {\em Stasheff associahedra} and {\em Stasheff multiplihedra}. The same holds for its $(\aaa,\m)$-colored version $(A_\infty^{\operatorname{Col}}, M_\infty^{\operatorname{Col}})$. Algebras over  $(A_\infty^{\operatorname{Col}}, M_\infty^{\operatorname{Col}})$ are pairs of an $A_\infty$-algebra and an $A_\infty$-module over it, and maps of such pairs are $A_\infty$ in both colors. The operadic pair $(\Omega,T)$ above is a certain quotient of $(A_\infty^{\operatorname{Col}}, M_\infty^{\operatorname{Col}})$. \\

On the polyhedral side, taking quotients corresponds to contraction. The contraction of associahedra corresponding to the projection $A_\infty^{\operatorname{Col}} \to \Omega$ was known from \cite{ACD}; it resulted in {\em cubes}. For the projection $M_\infty^{\operatorname{Col}} \to T$, the corresponding contraction was not previously known. In this paper we compute it and prove the following: \\

{\bf Theorem.} There exists an isomorphism of chain complexes $C_*(\mathcal{F}_n) \simeq  T(\aaa^n,\m;\m)$, where $\mathcal{F}_n$ are {\em} freehedra of \cite{San}. \\

 These polytopes were originally introduced  to study free loop spaces, and until now they bore no relation to operads. Therefore, the current paper establishes a dictionary between \cite{San} and \cite{ACD}. In particular, it seems that the freehedral diagonal of \cite{San} coincides with the diagonal $T \to T \boxtimes T$ of \cite{ACD}. In further research we expect to use polyhedral methods to define {\em weakly Hopf} structure on the operadic pair $(\Omega,T)$. This would provide a weakly monoidal structure on the DG-category $\operatorname{Rep}^h(G)$, giving a lift from the homotopy level.

\subsection{Organization of the paper}
This paper aims to be as self-contained as possible, thus the length.  In Section 2 we give an overview of operadic theory in one and two colors, and introduce operadic pairs. In Section 3 we present associahedra and multiplihedra. In section 4 we summarize the existing definitions of freehedra. In Section 5 we prove our main theorem, which provides an operadic meaning for freehedra. In Section 6 we discuss the existing projections betweeen polyhedral families in terms of operadic pairs. In Section 7 we define strictly Hopf operadic pairs and prepare the ground for studying weakly Hopf operadic pairs.

\subsection{Acknowledgements} This paper would not have been written without Jim Stasheff's advice and support. I am also grateful to Sergey Arkhipov, Ryszard Nest, Lars Hesselholt, Nathalie Wahl, Camilo Abad, Stephen Forcey and Samson Saneblidze for discussions and interest. Finally, I am grateful to Vladimir Dotsenko, Timothy Logvinenko and Svetlana Makarova for inviting me to present this research at their seminars.

\section{Operads and operadic pairs}
Let $\C$ be a closed monoidal category with sums. 

\begin{defi}
In the category ${\mathbb{N}}$-$\operatorname{Seq}(\C)$ of $\mathbb{N}$-sequences in $\C$, an object $\mathcal{P}$ is a collection of $\mathcal{P}(i) \in \C$ for $i \geq 1$. For $\mathcal{P}$ and $\mathcal{Q}$ in ${\mathbb{N}}$-$\operatorname{Seq}(\C)$, their tensor-product $\mathcal{P} \odot \mathcal{Q}$ is given by
$$ (\mathcal{P} \odot \mathcal{Q})(n) = \bigoplus_{i_1+ \ldots +i_k = n} \mathcal{P}(k) \otimes \mathcal{Q}(i_1) \otimes \ldots \otimes \mathcal{Q}(i_k)$$
\end{defi}

This makes ${\mathbb{N}}$-$\operatorname{Seq}(\C)$ a non-symmetric monoidal category. The unit is the ${\mathbb{N}}$-sequence $\underline{\mathbb{I}}$ with $\underline{\mathbb{I}}= \mathbb{I}$ and $\underline{\mathbb{I}}(n) = 0$ for $n \geq 1$, where $\mathbb{I}$ is the monoidal unit of $\C$ and $0$ is the initial object of $\C$.

\begin{defi}
An non-symmetric operad in $\C$ is a unital algebra in  ${\mathbb{N}}$-$\operatorname{Seq}(\C)$.
\end{defi}

If $\mathcal{P}$ is an operad, we say that $\mathcal{P}(k)$ is  the object of arity $k$ operations. Explicitly, the operadic structure on $\mathcal{P}$ is given by a collection of composition maps
$$ \circ_{i_1,\ldots,i_k} \colon \mathcal{P}(k) \otimes \mathcal{P}(i_1) \otimes \ldots \otimes \mathcal{P}(i_k) \to \mathcal{P}(i_1+ \ldots+ i_k)$$
satisfying associativity conditions. The existence of a unit allows us to express all such compositions through
$$ \circ_{i} \colon \mathcal{P}(k) \otimes \mathcal{P}(l) \to \mathcal{P}(k+l-1)$$

\begin{defi}
Every object $X \in \C$ gives rise to its operad $\underline{\operatorname{End}}_X$, with $\underline{\operatorname{End}}_X(n) = \underline{\operatorname{Hom}}_\C(X^{\otimes n},X)$ and with operadic structure coming from compositions.
\end{defi}

\begin{defi}
For an operad $\mathcal{P}$ and an object $X$, the structure of a $\mathcal{P}$-algebra on $X$ is a map of operads $\mathcal{P} \to \underline{\operatorname{End}}_X$. If $X$ and $Y$ are $\mathcal{P}$-algebras, their map in $\operatorname{Alg}(\mathcal{P})$ is a map $X \to Y$ such that for any $n$ the diagram below commutes: 
\[ \begin{tikzcd}[column sep=huge]
\mathcal{P}(n) \otimes X^{\otimes n} \arrow{r}{\id_{\mathcal{P}(n)} \otimes f^{\otimes n}} \arrow{d}{} & \mathcal{P}(n) \otimes Y^{\otimes n}  \arrow{d}{} \\%
X \arrow{r}{f} & Y
\end{tikzcd}
\]

\end{defi}

Let $\C$ be the category of chain complexes $\dgvect$. Operads in $\dgvect$ are called DG-operads. The simplest DG-operad is $Ass$, with $Ass(n) = k$ for any $n$. Algebras over $Ass$ are DG algebras. The key operad for this paper is a classical resolution of $Ass$ called $A_\infty$. For detailed discussion of $A_\infty$-formalism, see for example \cite{Kel}.

\begin{defi}
The DG-operad $A_\infty$ is generated by operations $\mu_n$ of arity $n$ and degree $2-n$ for $n \geq 2$, with differential
$$d(\mu_n) = \sum_{i+j+k = n} \mu_{i+1+k}(\id {\otimes^i} \otimes \mu_j \otimes \id^{\otimes j})$$
\end{defi}

$A_\infty$-algebras are homotopy-associative algebras with an explicit system of all the higher coherences.

\begin{rem}
The category of algebras over $A_\infty$ is {\em not} what people usually mean by the category of $A_\infty$-algebras. The problem is that $\operatorname{Alg}(A_\infty)$ doesn't have enough morphisms. A morphism $A \to B$ in $\operatorname{Alg}(A_\infty)$ has to strictly respect multiplication $\mu_2$ and all the higher operations. A true $A_\infty$-morphism $A \to B$ should respect multiplication $\mu_2$ only up to homotopy, and includes the data of all the higher coherences $A^{\otimes n} \xrightarrow{\operatorname{deg}1-n} B$.
\end{rem}

To combat this difficulty we use operadic bimodules, i.e. bimodules in the category ${\mathbb{N}}$-$\operatorname{Seq}(\C)$. Note that this is not a symmetric monoidal category, so left and right actions differ a lot. 

\begin{defi}
For objects $X, Y \in \C$, let $\underline{\operatorname{Hom}}_{X,Y}$ be an $\mathbb{N}$-sequence given by $\underline{\operatorname{Hom}}_{X,Y}(n) = \underline{\operatorname{Hom}}_\C(X^{\otimes n},Y)$. It has a natural structure of a right module over $\underline{\operatorname{End}}_X$ and of a left module over $\underline{\operatorname{End}}_Y$ given by compositions.
\end{defi}

Below we present the standard resolution of the trivial $Ass$-bimodule given by $Ass$ itself.

\begin{defi}
$M_\infty$ is a bimodule over $A_\infty$ generated by $f_n$ of arity $n$ and degree $1-n$ for $n \geq 0$, with differentials
$$d(f_n) = \sum f_{r+1+t} (\id^{\otimes r} \otimes \mu_s \otimes \id^{\otimes r}) + \sum \mu_r (f_{i_1}\otimes \ldots \otimes f_{i_r})$$
\end{defi}

\begin{prop}
Let $A$, $B$ be two $A_\infty$-algebras with structure maps $\alpha \colon A_\infty
 \to \underline{\operatorname{End}}_A$ and $\beta \colon A_\infty
 \to \underline{\operatorname{End}}_B$. Then any $A_\infty$-morphism $f \colon A \to B$ is given by a structure map $\phi \colon M_\infty \to \underline{\operatorname{Hom}}_{X,Y}$ of bimodules over $A_\infty$,
% \phi \colon M_\infty \to \underline{\operatorname{Hom}}_{X,Y}$, 
 where $\underline{\operatorname{Hom}}_{X,Y}$ is viewed as a bimodule over $A_\infty$ via restrictions along $\alpha$ and $\beta$.
 \end{prop}
 
Note that the composition of $A_\infty$-morphisms is induced by a map 
$$c \co M_\infty \to M_\infty \otimes_{A_\infty} M_\infty$$ which is given on generators by 
$$ c(f_n) = \sum f_i \otimes f_{n-i+1}$$
and the identity $A_\infty$-morphisms are induced by a map 
$$\epsilon \co M_\infty \to A_\infty$$ which is given on generators by
$$ \epsilon(f_n) = \begin{cases} \id & n = 1 \\ 0 & n>1 \end{cases}$$

This suggests the following new definition.

\begin{defi}
An {\em operadic pair} is a pair $(\mathcal{P},\mathcal{M})$ where $\mathcal{P}$ is an operad and $\mathcal{M}$ is a counital coalgebra in operadic bimodules over $\mathcal{P}$,  with comultiplication $c \co \mathcal{M} \to \mathcal{M} \otimes_{\mathcal{P}} \mathcal{M}$ and counit $\epsilon\co \mathcal{M} \to \mathcal{P}$.
\end{defi}

\begin{defi}
For an operadic pair $(\mathcal{P},\mathcal{M})$, an object of $\operatorname{Alg}(\mathcal{P},\mathcal{M})$ is just a $\mathcal{P}$-algebra. For two such objects $A$ and $B$, with structure maps $\chi_A: \mathcal{P} \to \underline{\operatorname{End}}_{A}$ and $\chi_B: \mathcal{P} \to \underline{\operatorname{End}}_{B}$, a morphism $f$ in $\operatorname{Alg}(\mathcal{P},\mathcal{M})$ is given by a structure map of $\mathcal{P}$-bimodules $\chi_f: \mathcal{M} \to \underline{\operatorname{Hom}}_{A,B}$. The composition is induced by $c$ and the identity morphisms are incduced by $\epsilon$.
\end{defi}

Then $(A_\infty,M_\infty)$ is an example of DG-operadic pair, and the category of $A_\infty$-algebras is precisely $\operatorname{Alg}(A_\infty,M_\infty)$. \\ 

\begin{rem}
Every operad $\mathcal{P}$ forms a counital coalgebra in bimodules over itself, resulting in a trivial operadic pair $(\mathcal{P},\mathcal{P})$. For this pair, we have $\operatorname{Alg}(\mathcal{P},\mathcal{P}) = \operatorname{Alg}(\mathcal{P})$.
\end{rem}

For an operadic pair $(\mathcal{P},\mathcal{M})$, by its underlying pair we mean the pair $(\mathcal{P},\mathcal{M})$ with forgotten coalgebra structure on $\mathcal{M}$. \\

We now repeat the story with colors. Fix the set of colors $\operatorname{Col}$.

\begin{defi}
In the category ${\mathbb{N}}$-$\operatorname{Seq}_{\operatorname{Col}}(\C)$ of colored $\mathbb{N}$-sequences in $\C$, an object $\mathcal{P}$ is a collection of $\mathcal{P}(c_1,\ldots,c_k;c) \in \C$ for all tuples $c_1,\ldots, c_k, c$ with $c_i$ and $c$ in $\operatorname{Col}$. Here, $c_i$ are called input colors and $c$ is called output color. For $\mathcal{P}$ and $\mathcal{Q}$ in ${\mathbb{N}}$-$\operatorname{Seq}_{\operatorname{Col}}(\C)$, their tensor-product $\mathcal{P} \odot \mathcal{Q}$ is given by
\begin{align*}
&    (\mathcal{P} \odot \mathcal{Q})(c_1, \ldots, c_n;c) = \\
&    \bigoplus_{\substack{i_1+ \ldots +i_k = n \\ c'_1, \ldots, c'_k \in \operatorname{Col}}} \mathcal{P}(c'_1,\ldots,c'_k;c) \otimes \mathcal{Q}(c_1,\ldots, c_{i_1};c'_1) \otimes \ldots \otimes \mathcal{Q}(c_{n-i_k+1},\ldots, c_n; c'_k)
\end{align*} 
\end{defi}

Colored operads and colored operadic bimodules are defined as algebras and bimodules in this new monoidal category. 

\begin{defi}
Let $\{X_c \}_{c \in \operatorname{Col}}$ be a collection of objects in $\C$. The colored operad $\underline{\operatorname{End}}_{\{X_c \}}$ is defined by 
$$\underline{\operatorname{End}}_{\{X_c \}}(c_1,\ldots,c_n;c) = \underline{\operatorname{Hom}}_{\C}(X_{c_1} \otimes \ldots \otimes X_{c_n},X_c)$$
with operadic structure given by compositions.
\end{defi}
\begin{defi}
An algebra over a colored operad $\mathcal{P}$ is a collection of objects $\{X_c \}_{c \in \operatorname{Col}}$ with a map of operads $\mathcal{P} \to \underline{\operatorname{End}}_{\{X_c \}}$. If $\{X_c \}$ and $\{Y_c \}$ are $\mathcal{P}$-algebras, then their map in $\operatorname{Alg}(\mathcal{P})$ is a collection of maps $f_c \co X_c \to Y_c$ such that for every tuple $(c_1,\ldots,c_n,c)$ the following diagram commutes.
\[ \begin{tikzcd}[column sep=3 cm]
\mathcal{P}(c_1,\ldots,c_n;c) \otimes \bigotimes_{i=1}^n X_{c_i}  \arrow{r}{\id \otimes \bigotimes_{i=1}^n f_{c_i}} \arrow{d}{} & \mathcal{P}(c_1,\ldots,c_n;c) \otimes \bigotimes_{i=1}^n Y_{c_i}  \arrow{d}{} \\%
X_c \arrow{r}{f_c} & Y_c
\end{tikzcd}
\]

\end{defi}

\begin{defi}
Let $\{X_c \}_{c \in \operatorname{Col}}$ and  $\{Y_c \}_{c \in \operatorname{Col}}$  be two collection of objects in $\C$. The  colored $\mathbb{N}$-sequence $\underline{\operatorname{Hom}}_{\{X_c \},\{Y_c \}}$ is defined by
$$\underline{\operatorname{Hom}}_{\{X_c \},\{Y_c \}} (c_1,\ldots,c_n;c) = \underline{\operatorname{Hom}}_{\C}(X_{c_1} \otimes \ldots \otimes X_{c_n},Y_c)$$
It has the natural structure of a left module over $\underline{\operatorname{End}}_{\{X_c \}}$ and a right module over $\underline{\operatorname{End}}_{\{Y_c \}}$ given by compositions.
\end{defi}

The definition of an operadic pair can now be repeated verbatim. \\

In the rest of the paper we will only be interested in the case when $\operatorname{Col} = \{\aaa,\m\}$, with $\aaa$ for {\em algebra} and $\m$ for {\em module}. The simplest example of a colored DG-operad is $Ass^{\operatorname{Col}}$, has $Ass^{\operatorname{Col}}(\aaa,\ldots,\aaa;\aaa) = k$, $Ass^{\operatorname{Col}}(\aaa,\ldots,\aaa,\m;\m) = k$ and $0$ everywhere else. An algebra over this colored operad is a pair $(A,M)$ where $A$ is a DG-algebra and $M$ is a DG-module over $A$. Similarly to the non-colored case, the operad $Ass^{\operatorname{Col}}$ has a standard resolution $A_\infty^{\operatorname{Col}}$.

\begin{defi}
$A_\infty^{\operatorname{Col}}$ is generated by operations $\mu_n^{\aaa} \in A_\infty^{\operatorname{Col}}(\aaa^n;\aaa)$ of degree $2-n$ and $\mu_n^{\m} \in A_\infty^{\operatorname{Col}}({\aaa}^{n-1},\m;\aaa)$ of degree $2-n$, with differentials
$$d(\mu_n^\aaa) = \sum_{i+j+k = n} \mu^\aaa_{i+1+k}(\id_\aaa^{\otimes^i} \otimes \mu^\aaa_j \otimes \id_\aaa^{\otimes j})$$
$$d(\mu_n^\m) =  \sum_{\substack{i+j+k = n \\ j \geq 1, k \geq 1}} \mu^\aaa_{i+1+k}(\id_\aaa^{\otimes^i} \otimes \mu^\aaa_j \otimes \id_\m^{\otimes j}) + \sum_{\substack{i+j = n \\ j \geq 1}} \mu_{}( \id_\aaa^{\otimes i} \otimes \mu_j^\m ) $$
\end{defi}

Again, the correct category of algebras is obtained via the formalism of operadic pairs. 

\begin{defi}
The operadic bimodule $M_\infty^{\operatorname{Col}}$ is generated over $A_\infty^{\operatorname{Col}}$ by $f^\aaa_n$ and $f^\m_n$, with differentials 
$$d(f_n^\aaa) = \sum f^\aaa_{r+1+t} (\id_\aaa^{\otimes r} \otimes \mu^\aaa_s \otimes \id_\aaa^{\otimes t}) + \sum \mu^\aaa_r (f^\aaa_{i_1}\otimes \ldots \otimes f^\m_{i_r})$$

\begin{align*}
& d(f^\m_n) = \sum f^\m_{r+1+t} (\id_\aaa^{\otimes r} \otimes \mu^\aaa_s \otimes \id_\m^{\otimes t}) + \sum \mu^\aaa_r (f^\aaa_{i_1}\otimes \ldots \otimes f^\m_{i_r}) + \\
& +  \sum f^\m_{r+1+t} (\id_\aaa^{\otimes r} \otimes \mu^\m_t) 
\end{align*}

The comultiplication $c\co M_\infty^{\operatorname{Col}} \to M_\infty^{\operatorname{Col}} \otimes_{ A_\infty^{\operatorname{Col}}} M_\infty^{\operatorname{Col}} $ is given on generators by 
$$ c(f^\aaa_n) = \sum f^\aaa_i \otimes f^\aaa_{n-i+1}$$
$$ c(f^\m_n) = \sum f^\m_i \otimes f^\m_{n-i+1}$$
The counit $\epsilon \co M_\infty^{\operatorname{Col}}  \to A_\infty^{\operatorname{Col}}$ is given on generators by 
$$ \epsilon(f^\aaa_n) = \begin{cases} \id_\aaa & n = 1 \\ 0 & n>1 \end{cases}$$
$$ \epsilon(f^\m_n) = \begin{cases} \id_\m & n = 1 \\ 0 & n>1 \end{cases}$$
This makes $(A_\infty^{\operatorname{Col}}, M_\infty^{\operatorname{Col}})$ an operadic pair.
\end{defi}

In this paper, we are mainly interested in a certain quotient of  $(A_\infty^{\operatorname{Col}}, M_\infty^{\operatorname{Col}})$.

\begin{defi} Let $\Omega$ be the quotient of $A_\infty^{\operatorname{Col}}$ by the ideal $I$ generated by all $\mu^\aaa_i$ for $i > 2$. Let $T$ be a further quotient of $M_\infty^{\operatorname{Col}}/I$ by a subbimodule generated by $f^\aaa_i$ for $i>1$. 
\end{defi}

$(\Omega,T)$ remains an operadic pair. \\

Albeit in a different language, the operaic pair $(\Omega, T)$ was closely studied in \cite{ACD} in connection to representations up to homotopy. There the authors developed a convenient forest notation for bases of $\Omega$ and $T$, which we use in the main theorem of this paper. \\

\begin{defi}
A \emph{short forest} is a sequence of planar trees of depth 2. Inner edges are called branches and outer edges are called leaves. For a short forest $F$, let $l(F)$ be the number of leaves, let $b(F)$ be the number of branches and let $t(F)$ be the number of trees.
\end{defi}

Below is an example of a short forest $F$ with $l(F) = 12$, $b(F) = 8$ and $t(F) = 5$. The roots are depicted as connected with a horizontal line, the ground.

\begin{center}
\begin{forest}
for tree = {grow'=90,circle, fill, minimum width = 4pt, inner sep = 0pt, s sep = 13pt}
[{},phantom
[{},name = one [[]] [[][]] ]
[{},name = two [[]] [[]] [[]] ]
[{},name = three [[][][][]] ]
[{},name = four [[]] ]
[{},name = five [[]] ]
]
\draw[black] (one) -- (two) -- (three) -- (four) -- (five);
\end{forest}
\end{center}

For a forest $F$, denote by $F^i$ its $i$-th tree. Write $F^i$ as $(F^i_1, \ldots, F^i_{b_i})$, where $F^i_j$ denoted the number of leaves on the $j$-th branch of $i$-th tree. 

\begin{prop}
\label{isom}
The basis of $\Omega (a^i,m;m)$ is given by short forests with $l(F)$, where the degree of the forest is $t(F)-b(F)$.
\end{prop}

\begin{proof}
To the tree $F^i = (F^i_1, \ldots, F^i_{b_i})$ we assign the operation
$$ \mu(F^i) = \mu^\m_{b_i}\left((\mu^\aaa_2)^{F^i_1-1},\ldots, (\mu^\aaa_2)^{F^i_{b_1}-1},\operatorname{id}^\m\right)$$

The powers of $\mu^\aaa_2$ are well-defined since  $\mu^\aaa_2$ is associative. We then build the operation for the whole forest by composing $\mu(F^i)$ for all the trees in the same order as the trees appear in the forest.
\end{proof}

Under this isomorphism, the example forest above corresponds to the operation 
$$\mu^\m_2 \Bigg(\id^\aaa, \mu^\aaa_2, \mu^\m_3 \bigg(\id^\aaa,\id^\aaa,\id^\aaa, \mu^\m_1 \Big((\mu^\aaa_2)^3, \mu^\m_1 \big( \id^\aaa, \mu^\m_1 \big) \Big) \bigg) \Bigg) $$

The differential of $\Omega$ in this basis can be described in terms of two forest transformations, $U$ (for "unite") and $S$ (for "separate"). Let $F$ be a forest with a chosen pair of branches $B = (B_l,B_r)$ belonging to the same tree $T$.

\begin{enumerate}
    \item  $U(F,B)$ is the forest where $B_l$ and $B_r$ are replaced with the one branch that has leaves of both $B_l$ and $B_r$.
    \item $S(F,B)$ is the forest where $T$ is replaced by two separate trees, $T_l$ with branches of $T$ up to $B_l$ and $T_r$ with branches of $T$ starting from $B_r$.
\end{enumerate}

For example, consider the following forest with $B = (B_l,B_r)$ highlighted green:

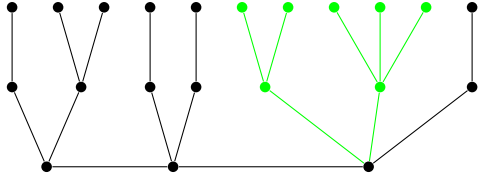
\begin{center}
\begin{forest}
for tree = {grow'=90,circle, fill, minimum width = 4pt, inner sep = 0pt, s sep = 13pt}
[{},phantom
[{},name = one [[]] [[][]] ]
[{},name = two [[]] [[]] ]
[{},name = three [ {}, for tree = {fill = green, edge = {color = green}} [][]] [ {}, for tree = {fill = green, edge = {color = green}} [] [] []] [[]]  ]
]
\draw[black] (one) -- (two) -- (three);
\end{forest}
\end{center}

Then $U(F,B)$ and $S(F,B)$ are the two forests below.

\begin{center}
\begin{forest}
for tree = {grow'=90,circle, fill, minimum width = 4pt, inner sep = 0pt, s sep = 13pt}
[{},phantom
[{},name = one [[]] [[][]] ]
[{},name = two [[]] [[]] ]
[{},name = three [  [][] [] [] []] [[]]  ]
]
\draw[black] (one) -- (two) -- (three);
\end{forest}
\end{center}

\begin{center}
\begin{forest}
for tree = {grow'=90,circle, fill, minimum width = 4pt, inner sep = 0pt, s sep = 13pt}
[{},phantom
[{},name = one [[]] [[][]] ]
[{},name = two [[]] [[]] ]
[{},name = three [  [][]]]
[{},name = four [ [] [] []] [[]]  ]
]
\draw[black] (one) -- (two) -- (three) -- (four);
\end{forest}
\end{center}

\begin{prop}
Under the correspondence of Prop. \ref{isom}, the differential of $\Omega$ is this:
$$d(F) = \sum_{B = (B_l,B_r)} \pm U(F,B) + \sum_{B = (B_l,B_r)} \pm S(F,B)$$
where in both sums $B$ runs along the set of neighbouring branch pairs. The operadic composition is given by forest concatenation when composing two-colored operations or by leaf multiplication when composing with a one-colored operation.

\end{prop}

We now explain a similar description for $T$.

\begin{prop}
\label{isom2}
The basis of $T(\aaa^i,\m;\m)$ is given by triples $(F,T,G)$ where $F$ and $G$ are forests, $T$ is a tree, and the total number of leaves is $i$. 
\end{prop}
\begin{proof}
To the tree $F^i = (F^i_1, \ldots, F^i_{b_i})$ in the left forest, we assign the operation
$$ \mu(F^i) = \mu^\m_{b_i}\left((\mu^\aaa_2)^{F^i_1-1},\ldots, (\mu^\aaa_2)^{F^i_{b_i}-1},\operatorname{id}_\m \right)$$

To the middle tree $T = (T_1, \ldots, T_t)$, we assign the operation
$$ \mu(T) = f^\m_{t}\left((\mu^\aaa_2)^{T_1-1},\ldots, (\mu^\aaa_2)^{T_t-1},\operatorname{id}_\m \right)$$

To the tree $G^i = (G^i_1, \ldots, G^i_{c_i})$ in the right forest, we assign the operation
$$ \mu(G^i) = \mu^\m_{c_i}\left((\mu^\aaa_2)^{G^i_1-1},\ldots, (\mu^\aaa_2)^{G^i_{c_i}-1},\operatorname{id}_\m \right)$$

We then build the operation for the whole triple by composing $\mu(F^i)$, $\mu(T)$ and $\mu(G^i)$ in the same order as the trees appear in the triple.

\end{proof}

Informally, the right forest $G$ is what happens before we map, the middle tree $T$ is the map itself, and the left forest $F$ is what happens after we map. Below is an example triple corresponding to $\mu^\m_1 \left( \id_\aaa,\mu^\m_1 \left( \mu^\aaa_2,f_2 \left( \mu^\aaa_2,\id_\m \right) \right ) \right)$.

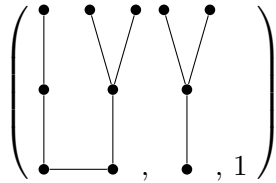
\[ \left ( \vcenter{\hbox{ \begin{forest}
for tree = {grow'=90,circle, fill, minimum width = 4pt, inner sep = 0pt, s sep = 13pt}
[{},phantom
[{},name = one [[]] ]
[{},name = two [[][]] ]
]
\draw[black] (one) -- (two);
\end{forest},
\begin{forest}
for tree = {grow'=90,circle, fill, minimum width = 4pt, inner sep = 0pt, s sep = 13pt}
[{},phantom
[  [[][]] ]
]
\end{forest}, 1  }} \right ) \]

To describe the differential, we need to modify the definition of the transformation $S$ in the case when it is applied to the branch pair in the middle tree, because two trees cannot both remain in the middle. Set $S_l( (F,T,G),B) = (F \circ T_l, T_r, G)$ and  $S_r( (F,T,G),B) = (F, T_l, T_r \circ G)$. Now for $B$ any neighbouring pair of branches is $(F,T,G)$, define
$$ S((F,T,G),B) = \begin{cases} (S(F,B), T, G) & B \subset F \\ S_l((F,T,G),B) + S_r((F,T,G),B)  & B \subset T \\ (F, T, S(G,B)) & B \subset G \end{cases} $$

\begin{prop}
Under the correspondence of Prop. \ref{isom2}, the differential of $T$ is this:
\begin{align*}
& d(F,T,G) = \\
& \sum_{B} \pm U((F,T,G),B) + \sum_{B} \pm S((F,T,G)B) + (F \circ T, 1, G) + (F, 1, T \circ G)
\end{align*}
where in both sums $B$ runs along the set of neighbouring branch pairs anywhere in the triple.  The operadic bimodule structure is given either by forest concatenation when composing with operations in $\Omega(\aaa^n,\m;\m)$ or by leaf multiplication when composing with operations in $\Omega(\aaa^n;\aaa)$.
\end{prop}

\section{Associahedra and multiplihedra}
\subsection{Associahedra}
It is a well known fact that the DG operad $A_\infty$ is obtained by the functor of cellular chains from a CW-operad of {\em Stasheff associahedra} (see \cite{Sta} and \cite{Tam}).

\begin{defi}
An abstract polytope $\mathcal{K}(n)$ has faces corresponding to planar trees with $n$ leaves. The face $T$ is a subface of the face $T'$ if $T'$ can be obtained from $T$ by contracting inner edges. Viewed as an $\mathbb{N}$-sequence in the category of CW-complexes, $\mathcal{K}$ has an operadic structure given by tree grafting.
\end{defi}

\begin{prop}
\label{assoc}
$C_*(\mathcal{K}) = A_\infty$. Under this isomorphism, the $n$-corolla corresponds to $\mu_n$.
\end{prop}

$\mathcal{K}(1)$ and $\mathcal{K}(2)$ are points. The pictures below show the interval $\mathcal{K}(3)$ and the pentagon $\mathcal{K}(4)$, with faces labelled by planar trees.

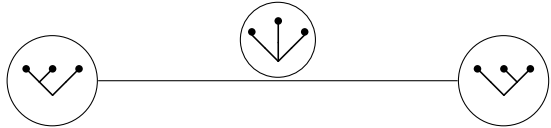
\begin{center}
\begin{tikzpicture}[
  vertex/.style={circle,draw,minimum size=8mm,inner sep=0pt, scale = 1.5},
  edge/.style={circle,draw,minimum size=6mm,inner sep=0pt, scale = 1.5}
  ]
  
 \node[vertex] (v0) at (0,0) {\RS{{L{lr}}{Rr}}};
 \node[vertex] (v1) at (6,0) {\RS{{Ll}{R{rl}}}};
 
 \draw (v0) -- node[edge,above=1pt]{\RS{{Ll}{Ii}{Rr}}} (v1);
 
\end{tikzpicture}
\end{center}

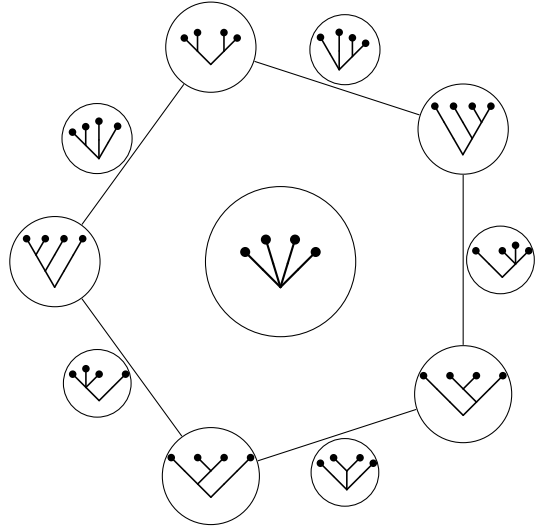
\begin{center}
\begin{tikzpicture}[
  vertex/.style={circle,draw,minimum size=8mm,inner sep=0pt, scale = 1.5},
  edge/.style={circle,draw,minimum size=6mm,inner sep=0pt, scale = 1.5}
  ]
  \newcommand\R{3cm}
  \node[circle,draw,scale = 2] at (0,0) {\RS{{Ll}{Qq}{Zz}{Rr}}};
  
  \node[vertex](v0) at (0*72+36:\R) {\RS{{MMm}{S{Mm}S{ms}}}};
  \node[vertex](v1) at (1*72+36:\R) {\RS{{L{li}}{R{ir}}}};
  \node[vertex](v2) at (2*72+36:\R) {\RS{{SSs}{M{Ss}M{sm}}}};
  \node[vertex](v3) at (3*72+36:\R) {\RS{{L{Ll}{R{lr}}}{RRr}}};
  \node[vertex](v4) at (4*72+36:\R) {\RS{{R{Rr}{L{rl}}}{LLl}}};
  \draw (v0) edge node[edge,above=1pt,xshift=2pt]{\usebox\sba} (v1);
  \draw (v1) edge node[edge,left,yshift=4pt]{\usebox\sbb} (v2);
  \draw (v2) -- node[edge,left,yshift=-4pt]{\usebox\sbc} (v3);
  \draw (v3) -- node[edge,below=1pt,xshift=2pt]{\usebox\sbd} (v4);
  \draw (v4) -- node[edge,right=1pt]{\usebox\sbe} (v0);
\end{tikzpicture}
\end{center}

It is a straigtforward observation that the $(\aaa,\m)$-colored operad $A_\infty^{\operatorname{Col}}$ can also be obtained from associahedra via cellular chains. Precisely, let $\mathcal{K}^{\operatorname{Col}}$ be a colored CW-operad with \begin{align*}
& \mathcal{K}^{\operatorname{Col}}(\aaa^n;\aaa) = \mathcal{K}(n);\\
& \mathcal{K}^{\operatorname{Col}}(\aaa^{n-1},\m;\m) = \mathcal{K}(n);\\
& \emptyset \text{ elsewhere}.
\end{align*}
Then $C_*(\mathcal{K}^{\operatorname{Col}}) = A_\infty^{\operatorname{Col}}$, with the $n$-corolla of $ \mathcal{K}^{\operatorname{Col}}(\aaa^n;\aaa)$ corresponding to $\mu_n^\aaa$ and with the $n$-corolla of $\mathcal{K}^{\operatorname{Col}}(\aaa^{n-1},\m;\m)$ corresponding to $\mu_n^\m$. \\

\subsection{Multiplihedra} $M_\infty$, the operadic bimodule over $A_\infty$, is also obtained by the functor of cellular chains from polytopes $\mathcal{J}$ called {\em multiplihedra} that form  a CW-operadic bimodule over $\mathcal{K}$. According to \cite{For}, multiplihedra admit a description in terms of trees, similar to the description of associahedra.

\begin{defi}
A painted planar tree $T$ is a planar tree with a possibility of single-input vertices, and with a selected subtree $T_{\operatorname{painted}}$ such that
\begin{itemize}
    \item the root of $T$ belongs to $T_{\operatorname{painted}}$
    \item the leaves of $T$ do not belong to $T_{\operatorname{painted}}$
    \item every single-input vertex of $T$ is a leaf of $T_{\operatorname{painted}}$
    \item for every vertex of $T_{\operatorname{painted}}$ either all inputs are in $T_{\operatorname{painted}}$ or all inputs are not in $T_{\operatorname{painted}}$
\end{itemize}
\end{defi}

The picture below shows some examples of such painted trees.

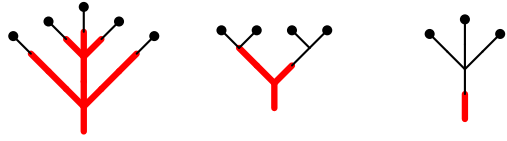
\begin{center}
\begin{tikzpicture}
\node[scale = 2]{\usebox\rbo};
\node[right = 2cm, scale = 2]{\usebox\mbo};
\node[left = 2cm, scale = 2]{\usebox\lbo};
\end{tikzpicture}
\end{center}

\begin{defi}
For a painted tree $T$, the admissible contractions are: 
\begin{enumerate}
    \item contract an inner edge of $T$ that is unpainted. For example, $$ \RS{BB {LL{Ll}{Rr}} {RRRr}} \longrightarrow \RS{BB {LLl}{IIi}{RRr}}$$
    \item contract an edge that is inner to $T_{\operatorname{painted}}$. For example,
     $$ \RS{BB {AA{Al}{Cr}} {CCRr}} \longrightarrow \RS{BB {AAl}{BBi}{CCr}}$$
    \item contract a corolla of painted leaves. For example,
    $$ \RS{BB {AA{Al}{Cr}} {CCRr}} \longrightarrow \RS{BB {AA{Ll}{Rr}} {CCRr}}$$
\end{enumerate}
\end{defi}

\begin{defi}
An abstract polytope $\mathcal{J}(n)$ has faces corresponding to painted planar trees. The face $T$ is a subface of the face $T'$ if $T'$ can be obtained from $T$ by a sequence of admissible contractions. Operadic bimodule structure is again given by tree grafting. For left module structure, the formerly unpainted tree remains unpainted, and for right module structure, the formerly unpainted tree admits the maximal painting.
\end{defi}

Below are examples of left and right grafting:
$$ \RS{II {LLl} {RRr}} \circ_1 \RS{BB {AA{Al}{Cr}} {CCRr}} = \RS{BB {AA  {AA{Al}{Cr}} {CCRr} } {CCCCRr} }
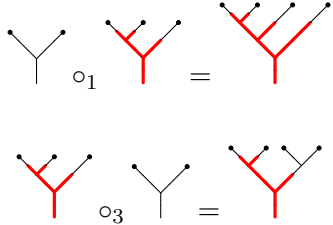
$$
$$\RS{BB {AA{Al}{Cr}} {CCRr}} \circ_{3} \RS{II {LLl} {RRr}} = \RS{BB {AAA{Al}{Cr}} {CCR{Ll} {Rr}}}$$

The picture below illustrates the interval $\mathcal{J}(2)$ and the hexagon $\mathcal{J}(3)$, with faces labelled by colored trees.

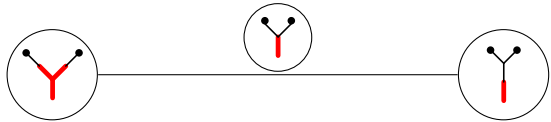
\begin{center}
\begin{tikzpicture}[
  vertex/.style={circle,draw,minimum size=8mm,inner sep=0pt, scale = 1.5},
  edge/.style={circle,draw,minimum size=6mm,inner sep=0pt, scale = 1.5}
  ]
  
 \node[vertex] (v0) at (0,0) {\RS{B{Al}{Cr}}};
 \node[vertex] (v1) at (6,0) {\RS{BI{l}{r}}};
 
 \draw (v0) -- node[edge,above=1pt]{\RS{B{l}{r}}} (v1);
 
\end{tikzpicture}
\end{center}

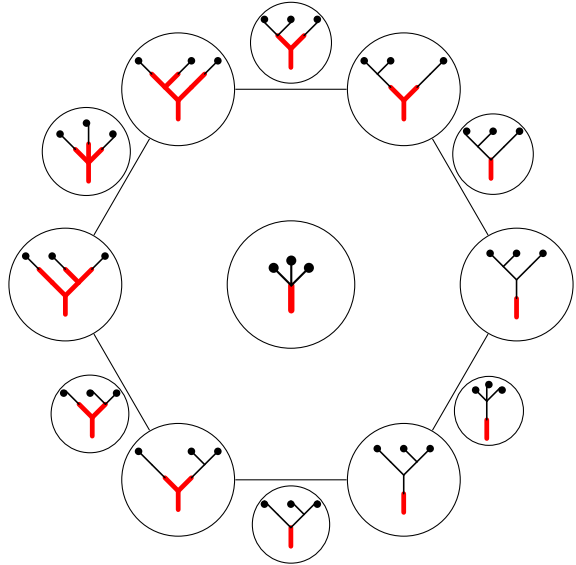
\begin{center}
\begin{tikzpicture}[
  vertex/.style={circle,draw,minimum size=10mm,inner sep=0pt, scale = 1.5},
  edge/.style={circle,draw,minimum size=6mm,inner sep=0pt, scale = 1.5}
  ]
  \newcommand\R{3cm}
  \node[circle,draw,scale = 2] at (0,0) {\RS{Blir}};
  
  \node[vertex](v0) at (0:\R) {\RS{BI{L{l}{r}}{Rr}}};
  \node[vertex](v1) at (60:\R) {\RS{B{AL{l}{r}}{CRr}}};
  \node[vertex](v2) at (120:\R) {\RS{B{A{Al}{Cr}}{CCr}}};
  \node[vertex](v3) at (180:\R) {\RS{B{AAl}{C{Al}{Cr}}}};
  \node[vertex](v4) at (240:\R) {\RS{B{ALl}{CR{lr}}}};
  \node[vertex](v5) at (300:\R) {\RS{BI{Ll}{R{lr}}}};
  
  \draw (v0) -- node[above right = 2pt, edge]{\RS{B{L{l}{r}}{Rr}}} (v1);
  \draw (v1) -- node[above = 2pt, edge]{\RS{B{A{l}{r}}{Cr}}} (v2);
  \draw (v2) -- node[above left = 2pt, edge] {\RS{B{Al}{Bi}{Cr}}} (v3);
  \draw (v3) -- node[below left = 2pt, edge] {\RS{B{Al}{C{lr}}}}(v4);
  \draw (v4) -- node[below = 2pt,edge]{\RS{B{Ll}{R{lr}}}} (v5);
  \draw (v5) -- node[below right = 2pt,edge]{\RS{BIlir}} (v0);
 
\end{tikzpicture}
\end{center}

Let $C_n$ denote the painted tree labelling the top-dimensional cell of $\mathcal{J}(n)$. For example, $C_3 = \RS{Blir}$.

\begin{prop} The isomorphism of Prop. \ref{assoc} extends to
$C_*(\mathcal{K},\mathcal{J}) = (A_\infty, M_\infty)$, where by $(A_\infty, M_\infty)$ we mean just the underlying pair. Under this isomorphism, the corolla $C_i$ corresponds to $f_i$.
\end{prop}

\begin{rem}
$(\mathcal{K},\mathcal{J})$ does not form an CW-operadic pair because the map $c \co  C_*(\mathcal{J}) \to C_*(\mathcal{J}) \otimes_{C_*(\mathcal{K})} C_*(\mathcal{J})$ involves sums, and one cannot add maps of CW-complexes. In general, the notion of operadic pairs doesn't seem to be well adapted for non-additive categories like $Top$. However, it is often useful to realize the underlying pair of a DG-operadic pair as cellular chains on a CW-operad with a CW-bimodule.
\end{rem}

Similarly to the case of associahedra, we observe that the $(\aaa,\m)$-colored bimodule $M_\infty^{\operatorname{Col}}$ over $A_\infty^{\operatorname{Col}}$ can also be obtained from multiplihedra via cellular chains. Precisely, let $\mathcal{J}^{\operatorname{Col}}$ be a colored CW-sequence with \begin{align*}
& \mathcal{J}^{\operatorname{Col}}(\aaa^n;\aaa) = \mathcal{J}(n);\\
& \mathcal{J}^{\operatorname{Col}}(\aaa^{n-1},\m;\m) = \mathcal{J}(n);\\
& \emptyset \text{ elsewhere}.
\end{align*}

\subsection{Contraction problem}
In the category $\dgvect$ there is a projection of operadic pairs $(A_\infty^{\operatorname{Col}},M_\infty^{\operatorname{Col}}) \to (\Omega,T)$. On the polyhedral side, this should correspond to some contraction of associahedra and multiplihedra.  \\

The picture below illustrates how a pentagon $\mathcal{K}(4)$ contracts to a square $I^2$, if we remove the non-associativity of the algebra. For readability we label vertices not with binary trees but with expressions in 4 letters.

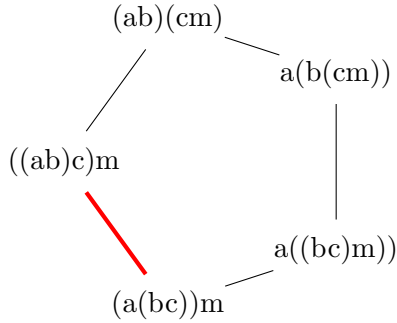
\begin{center}
\begin{tikzpicture}[
  vertex/.style={minimum size=8mm,inner sep=0pt}]
  
  \newcommand\R{2}
  
  \node[vertex](v0) at (0*72+36:\R) {a(b(cm))};
  \node[vertex](v1) at (1*72+36:\R) {(ab)(cm)};
  \node[vertex](v2) at (2*72+36:\R) {((ab)c)m};
  \node[vertex](v3) at (3*72+36:\R) {(a(bc))m};
  \node[vertex](v4) at (4*72+36:\R) {a((bc)m))};
  
  \draw (v3) -- (v4) -- (v0) -- (v1) -- (v2);
  \draw[ultra thick, red] (v2) -- (v3);
\end{tikzpicture}
\end{center}

The polyhedral contraction behind $A_\infty^{\operatorname{Col}} \to \Omega$ was computed, albeit in a different language, in \cite{ACD}.

\begin{prop}
\label{cubes}
$\Omega(\aaa^n,\m;\m) \simeq C_*(I^{n-1})$, and the projection  $A_\infty^{\operatorname{Col}} \to \Omega$ comes from a projection of associahedra to cubes.
\end{prop}

\begin{proof}
For a cube $I^{n-1}$, every face can be written as a word in letters $a$, $b$ and $c$, where $a$
 is interpreted as $\{0\}$, $b$ is interpreted as $[0,1]$, $c$
 is interpreted as $\{1\}$, and the word is interpreted as their product. For example, for the square the top-dimensional cell is $bb$, the initial vertex is $aa$, and the right side is $cb$. Now, having a short forest, you form the word by setting its $i$th letter equal to 
 \begin{itemize}
     \item $a$, if the leaves with numbers $i$ and $i+1$ belong to the same branch 
     \item $b$, if the leaves with numbers $i$ and $i+1$ belong to different branches of the same tree
     \item $c$, if the leaves with numbers $i$ and $i+1$ belong to different trees.
 \end{itemize}
 \end{proof}

Note that the above isomorphisms actually arrange the cubes into a CW-operad. \\

The corresponding contraction of multiplihedra was not previously known, and its computation is the goal of the current paper. The picture below illustrates the two-dimensional case, where a hexagon contracts to a pentagon. Warning: this pentagon is not an associahedron, but actually a freehedron.

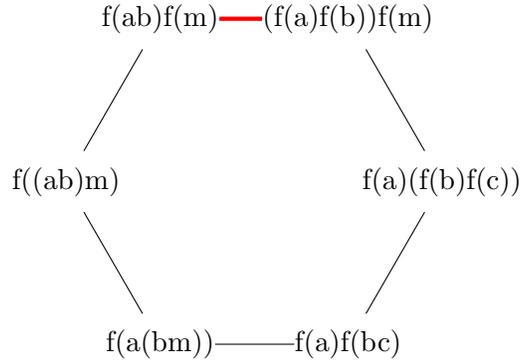
\begin{center}
\begin{tikzpicture}[
  vertex/.style={minimum size=8mm,inner sep=0pt}]

  \newcommand\R{2.5cm}

  \node[vertex](v0) at (0:\R) {f(a)(f(b)f(c))};
  \node[vertex](v1) at (60:\R) {(f(a)f(b))f(m)};
  \node[vertex](v2) at (120:\R) {f(ab)f(m)};
  \node[vertex](v3) at (180:\R) {f((ab)m)};
  \node[vertex](v4) at (240:\R) {f(a(bm))};
  \node[vertex](v5) at (300:\R) {f(a)f(bc)};
  
  \draw (v0) -- (v1) -- (v2) -- (v3) -- (v4) -- (v5) -- (v0);
  \draw[ultra thick, red] (v1) -- (v2);
\end{tikzpicture}
\end{center}

\section{Freehedra}
In this section I present freehedra directly following \cite{San} and \cite{RS}. Consequently I do not include any proofs, but instead include a lot of details and pictures. There are three definitions: as truncations of simplices, as subdivisions of cubes, and a purely combinatorial one. The first definition is not used in the main arguments of this paper, so the reader can safely skip it.

\subsection{Freehedra as truncations of simplices}
The first way to obtain freehedra is to cook them from simplices by applying two sequences of truncations. \\

Consider the simplex $\Delta^n$ in your favourite embedding to $\mathbb{R}^n$. We now define the first sequence of truncations. Let the original vertices be labelled $0$, $1$, $\ldots$, $n$. After each truncation, some new vertices are cut from edges by the truncating hyperplane; the vertex cut from the edge $a \to b$ is denoted $(ab)$. \\

\begin{enumerate}
    \item Let $Q_0$ be a hyperplane that separates $0$ from the other vertices. Remove everything connected to $0$. The resulting object is a simplicial prism. Its first simplicial face $S_1$ has vertices $(01), \ldots, (0n)$, and the second simplicial face $S_2$ has vertices $1, \ldots, n$.
    \item The second hyperplane is like $Q_0$ but for the $(n-1)$-simplices $S_1$ and $S_2$ simultaneously. It separates $(01)$ and $1$ from the other vertices. Denote it by $Q_1$ and remove everything connected to $(01)$ and $1$. \\
%    \item Similarly, $Q_2$ separates $((01)(02))$, $(02)$, $(12)$ and $2$ from the other vertices. Remove everything at the side of $((01)(02))$, $(02)$, $(12)$ and $2$.
\end{enumerate}

To define all the truncations inductively, denote by $L(k)$ the set of vertices that $Q_k$ separates from the rest. We see that $L(0) = \{0 \}$ and $L(1) = \{(01), 1 \}$. Now, having an expression for a vertex $v \in L(i-1)$, let $l_i(v)$ be the same expression with $i-1$ replaced by $i$. For example, $l_2((01)) = (02)$. Now $L(i)$ is defined to consist of vertices $l_i(v)$ and $(vl_i(v))$ for all $v \in L(i-1)$. This defines $Q(i)$, and we proceed to the next step by removing everything at the side of $L(i)$. The final truncation is by $Q_{n-2}$. We leave it to the interested reader to verify that this sequence of truncations is well-defined. \\

The second sequence is the same but starting at $n$ instead of $0$. Denote the hyperplanes by $P_0$, $\ldots$, $P_{n-2}$.\\

The pictures below show $\mathcal{F}_2$ and $\mathcal{F}_3$ cut out of a triangle and a tetrahedron respectively. Note that applying only one of the two truncation sequences yields cubes.

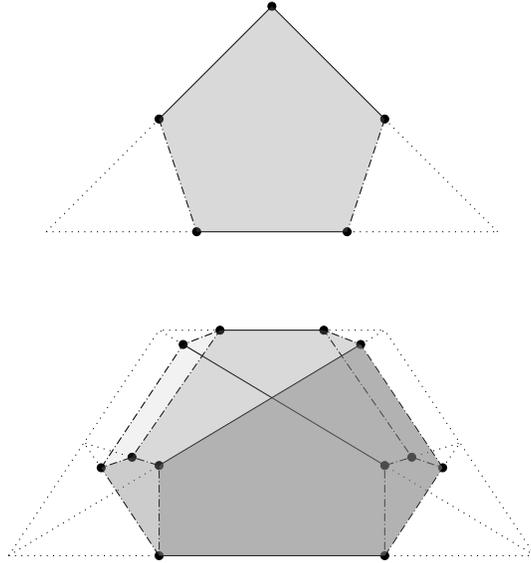
\begin{center}
\begin{tikzpicture}
\draw[thin] (1.5,1.5) -- (3,3) -- (4.5,1.5);
\draw[densely dashdotted] (1.5,1.5) -- (2,0);
\draw[densely dashdotted] (4.5,1.5) -- (4,0);
\draw[dotted] (0,0) -- (1.5,1.5);
\draw[dotted] (0,0) -- (2,0);
\draw[thin] (2,0) -- (4,0);
\draw[dotted] (4,0) -- (6,0);
\draw[dotted] (6,0) -- (4.5,1.5);

\filldraw[black] (1.5,1.5) circle (1.5pt);
\filldraw[black] (3,3) circle (1.5pt);
\filldraw[black] (4.5,1.5) circle (1.5pt);
\filldraw[black] (2,0) circle (1.5pt);
\filldraw[black] (4,0) circle (1.5pt);

\fill[gray,opacity = 0.3] (2,0) -- (1.5,1.5) -- (3,3) -- (4.5,1.5) -- (4,0) -- cycle;
\end{tikzpicture}
\end{center}
\vskip 1cm
\begin{center}
\begin{tikzpicture}

\draw[dotted] (0,0) -- (2,3);
\draw[dotted] (0,0) -- (2,1.2);
\draw[dotted] (0,0) -- (2,0);

\draw[dotted] (1.23,1.17) -- (1,1.5) -- (1.64,1.31);
\draw[dotted] (2.32,2.81) -- (2,3) -- (2.81,3);

\draw[densely dashdotted] (2,0) -- (1.23,1.17) -- (1.64,1.31) --(2,1.2) -- cycle;

\draw[densely dashdotted] (1.23,1.17) -- (2.32,2.81) -- (2.81,3) -- (1.64,1.31);

\filldraw[black] (1.23,1.17) circle (1.5pt);
\filldraw[black] (2,0) circle (1.5pt);
\filldraw[black]  (1.64,1.31) circle (1.5pt);
\filldraw[black]  (2.81,3) circle (1.5pt);
\filldraw[black] (2.32,2.81) circle (1.5pt);
\filldraw[black]  (2,1.2) circle (1.5pt);

\draw[dotted] (7-0,0) -- (7-2,3);
\draw[dotted] (7-0,0) -- (7-2,1.2);
\draw[dotted] (7-0,0) -- (7-2,0);

\draw[dotted] (7-1.23,1.17) -- (7-1,1.5) -- (7-1.64,1.31);
\draw[dotted] (7-2.32,2.81) -- (7-2,3) -- (7-2.81,3);

\draw[densely dashdotted] (7-2,0) -- (7-1.23,1.17) -- (7-1.64,1.31) --(7-2,1.2) -- cycle;

\draw[densely dashdotted] (7-1.23,1.17) -- (7-2.32,2.81) -- (7-2.81,3) -- (7-1.64,1.31);

\filldraw[black] (7-1.23,1.17) circle (1.5pt);
\filldraw[black] (7-2,0) circle (1.5pt);
\filldraw[black]  (7-1.64,1.31) circle (1.5pt);
\filldraw[black]  (7-2.81,3) circle (1.5pt);
\filldraw[black] (7-2.32,2.81) circle (1.5pt);
\filldraw[black]  (7-2,1.2) circle (1.5pt);

\draw[thin] (2,0) -- (5,0);
\draw[thin] (2,1.2) -- (7-2.32,2.81);
\draw[thin] (7-2.81,3) -- (2.81,3);
\draw[thin] (2.32,2.81) --  (5,1.2);

\fill[gray,opacity = 0.6] (2,0) -- (2,1.2) -- (7-2.32,2.81) -- (7-1.23,1.17) -- (5,0) -- cycle;

\fill[gray,opacity = 0.4] (2,0) -- (2,1.2) -- (1.64,1.31) -- (1.23,1.17) -- cycle;

\fill[gray,opacity = 0.1] (1.64,1.31) -- (1.23,1.17) -- (2.32,2.81) -- (2.81,3) -- cycle;

\fill[gray,opacity = 0.3] (2,1.2) -- (1.64,1.31) -- (2.81,3) -- (7 -2.81,3) -- (7-2.32,2.81) -- cycle;

\end{tikzpicture}    
\end{center}

\begin{prop}
Freehedra have two natural projections onto cubes and one natural projection onto simplices.
\end{prop}
\begin{proof}
All the three projections are obtained by de-truncation. 
\end{proof}

\subsection{Freehedra as subdivisions of cubes}
The second definition of freehedra is inductive. According to it, each freehedron $\mathcal{F}_n$ is a certain subdivision of $\mathcal{F}_{n-1}\times [0,1]$, thus all the freehedra arise as drawn on cubes. \\

We will first present a simplified version of this definition. At each step, the freehedron $\mathcal{F}_n$ will have a distinguished hyperface face $X_n$. These distinguished faces are only needed for user-friendliness; in the full version of the  definition, at  each step Saneblidze keeps track of labels for all hyperfaces.

\begin{defi}
Let $\mathcal{F}_0$ be the point, and let $\mathcal{F}_1$ be the interval $[0,1]$ with distinguished vertex $X_1$ = $1$. Assume $\mathcal{F}_{n-1}$ and its distinguihed face $X_{n-1}$ are defined. Consider the polyhedron $F_{n-1} \times [0,1]$, and split its hyperface $X_{n-1} \times [0,1]$ vertically into $X_{n-1} \times [0,\frac{1}{2}]$ and $X_{n-1} \times [\frac{1}{2},1]$. This is $\mathcal{F}_n$. Set $X_n = X_{n-1} \times [\frac{1}{2},1]$.
\end{defi}

The picture below illustrates freehedra in dimensions 1, 2 and 3. Distinguished hyperfaces are highlighted red.

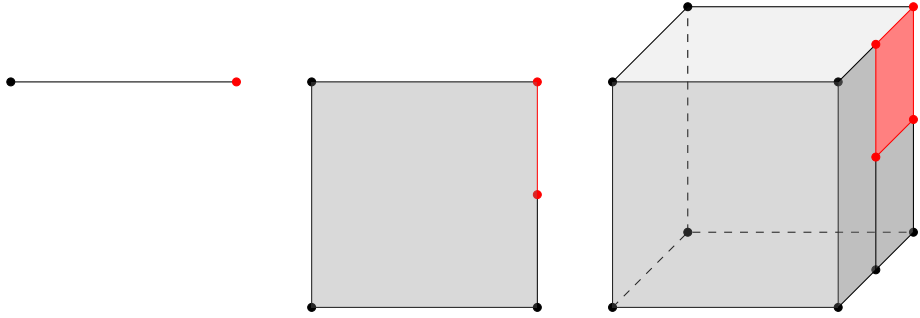
\begin{center}
\begin{tikzpicture}
\draw[thin] (0,0) -- (3,0);
\filldraw[black] (0,0) circle (1.5pt);
\filldraw[red] (3,0) circle (1.5pt);

\draw[thin] (3+4,0) -- (0+4,0) -- (0+4,-3) -- (3+4,-3) -- (3+4,-1.5);
\draw[red] (3+4,-1.5) -- (3+4,-0);
\filldraw[black] (0+4,0) circle (1.5pt);
\filldraw[black] (0+4,-3) circle (1.5pt);
\filldraw[red] (3+4,0) circle (1.5pt);
\filldraw[red] (3+4,-1.5) circle (1.5pt);
\filldraw[black] (3+4,-3) circle (1.5pt);

\fill[gray, opacity = 0.3] (0+4,0) -- (3+4,0) -- (3+4,-3) -- (0+4,-3) -- cycle;

\draw[thin] (8,0) rectangle (11,-3);
\draw[dashed] (8,-3) -- (9,-2) -- (12,-2);
\draw[dashed] (9,-2) -- (9,1);
\filldraw[black] (8,0) circle (1.5pt);
\filldraw[black] (8,-3) circle (1.5pt);
\filldraw[black] (11,0) circle (1.5pt);
\filldraw[black] (11,-3) circle (1.5pt);
\filldraw[black] (9,1) circle (1.5pt);
\filldraw[black] (9,-2) circle (1.5pt);
\filldraw[black] (11.5,-2.5) circle (1.5pt);
\filldraw[black] (12,-2) circle (1.5pt);

\fill[gray,opacity = 0.3] (8,0) rectangle (11,-3);
\fill[gray,opacity = 0.1] (8,0) -- (9,1) -- (12,1) -- (11,0);
\fill[gray,opacity = 0.5] (11,-3) -- (11,0) -- (11.5,0.5) -- (11.5,-2.5) -- cycle;
\fill[gray,opacity = 0.5] (11.5,-2.5) -- (11.5,-1) -- (12,-0.5) -- (12,-2) -- cycle;

\draw[thin] (8,0) -- (9,1) -- (12,1);
\draw[thin] (11,0) -- (11.5,0.5);
\draw[red] (11.5,0.5) -- (12,1) -- (12,1-1.5) --(11.5,-1) -- cycle;
\fill[red, opacity = 0.5] (11.5,0.5) -- (12,1) -- (12,1-1.5) --(11.5,-1) -- cycle;
\draw[thin] (11,-3) -- (11.5,-2.5) -- (11.5,-1);
\draw[thin] (11.5,-2.5) -- (12,-2) -- (12,-0.5);

\filldraw[red] (12,1) circle (1.5pt);
\filldraw[red] (12,-0.5) circle (1.5pt);
\filldraw[red] (11.5,0.5) circle (1.5pt);
\filldraw[red] (11.5,-1) circle (1.5pt);
\end{tikzpicture}
\end{center}

It is useful to have labels for all the hyperfaces. For $\mathcal{F}_n$, the labels are $d^0_i$ for $1 \leq i \leq n$, $d^1_i$ for $2 \leq i \leq n$ and $d^2_i$ for $1 \leq i \leq n$. The previosly defined distinguished hyperface is labelled $d_n^2$. The assignment is again given by an inductive procedure. For $1 \leq i \leq n$ and $\epsilon \in \{0,1\}$, let $e_i^0$ denote the face of the cube $[0,1]^n$ with coordinates $(x_1, \ldots, x_{i-1}, \epsilon, x_{i+1}, \ldots, x_n)$. For $\mathcal{F}_1$ label the vertex 0 by $d^0_1$ and label the vertex $1$ by $d^2_1$. Now assume that all the hyperfaces of $\mathcal{F}_{n-1}$ are labelled. Then hyperfaces of $\mathcal{F}_n$ viewed as a subdivision of $[0,1]^n$ are labelled according to the following table:\\

\begin{center}
 \begin{tabular}{|c|c|} 
 \hline
 {\bf Face in $\mathcal{F}_{n-1} \times [0,1]$} & {\bf Label in $\mathcal{F}_n$} \\ 
 \hline
 $e^0_i$, $1 \leq i \leq n$ & $d^0_i$  \\ 
 \hline
  $e^1_i$, $2 \leq i \leq n$ & $d^1_i$  \\ 
  \hline
  $d^2_i \times [0,1]$, $1 \leq i \leq n-2$ & $d^2_i$  \\
    \hline
  $d^2_{n-1} \times [0,\frac{1}{2}]$ & $d^2_{n-1}$  \\
    \hline
    $d^2_{n-1} \times [\frac{1}{2},1]$ & $d^2_{n}$  \\ 
 \hline
\end{tabular}
\end{center}

The picture below illustrates the labels for $\mathcal{F}_2$ and $\mathcal{F}_3$, both in their cubical and simplicial incarnations. The colors in dimension 3 are simply for user-friendliness.

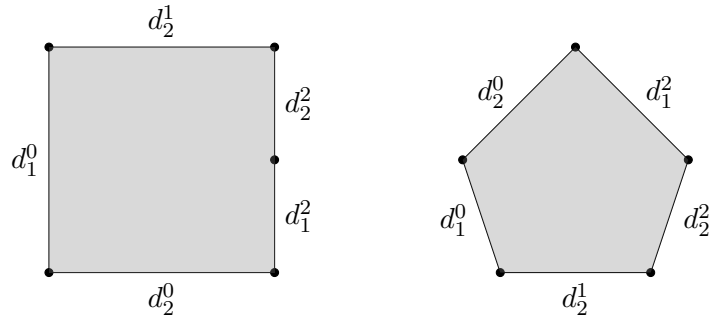
\begin{center}
\begin{tikzpicture}
\draw[thin] (0,0) rectangle (3,-3);
\filldraw[black] (0,0) circle (1.5pt);
\filldraw[black] (0,-3) circle (1.5pt);
\filldraw[black] (3,0) circle (1.5pt);
\filldraw[black] (3,-1.5) circle (1.5pt);
\filldraw[black] (3,-3) circle (1.5pt);

\fill[gray, opacity = 0.3] (0,0) -- (3,0) -- (3,-3) -- (0,-3) -- cycle;

\draw[thin] (1.5+4,1.5-3) -- (3+4,3-3) -- (4.5+4,1.5-3);
\draw[thin] (1.5+4,1.5-3) -- (2+4,0-3);
\draw[thin] (4.5+4,1.5-3) -- (4+4,0-3);
\draw[thin] (2+4,0-3) -- (4+4,0-3);

\node[anchor = north] at (7,-3) {$d^1_2$};
\node[anchor = east] at (5.7,-2.3) {$d^0_1$};
\node[anchor = east] at (6.2,-0.6) {$d^0_2$};
\node[anchor = west] at (8.3,-2.3) {$d^2_2$};
\node[anchor = west] at (7.8, -0.6) {$d^2_1$};

\filldraw[black] (1.5+4,1.5-3) circle (1.5pt);
\filldraw[black] (3+4,3-3) circle (1.5pt);
\filldraw[black] (4.5+4,1.5-3) circle (1.5pt);
\filldraw[black] (2+4,0-3) circle (1.5pt);
\filldraw[black] (4+4,0-3) circle (1.5pt);

\fill[gray,opacity = 0.3] (2+4,0-3) -- (1.5+4,1.5-3) -- (3+4,3-3) -- (4.5+4,1.5-3) -- (4+4,0-3) -- cycle;

\node[anchor = east] at (0,-1.5) {$d^0_1$};
\node[anchor = north] at (1.5,-3) {$d^0_2$};
\node[anchor = south] at (1.5,0) {$d^1_2$};
\node[anchor = west] at (3,-0.75) {$d^2_2$};
\node[anchor = west] at (3,-0.75-1.5) {$d^2_1$};

\end{tikzpicture}
\end{center}

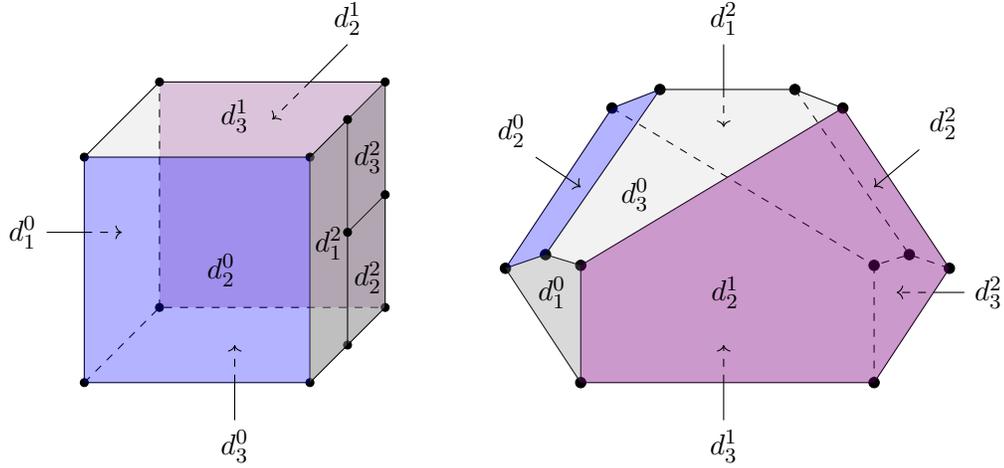
\begin{center}
\begin{tikzpicture}
\fill[violet, opacity = 0.2] (1,1) rectangle (4,-2);

\draw[thin] (0,0) rectangle (3,-3);
\draw[dashed] (0,-3) -- (1,-2) -- (4,-2);
\draw[dashed] (1,-2) -- (1,1);
\filldraw[black] (0,0) circle (1.5pt);
\filldraw[black] (0,-3) circle (1.5pt);
\filldraw[black] (3,0) circle (1.5pt);
\filldraw[black] (3,-3) circle (1.5pt);
\filldraw[black] (1,1) circle (1.5pt);
\filldraw[black] (1,-2) circle (1.5pt);
\filldraw[black] (3.5,-2.5) circle (1.5pt);
\filldraw[black] (4,-2) circle (1.5pt);

\fill[blue,opacity = 0.3] (0,0) rectangle (3,-3);
\fill[gray,opacity = 0.1] (0,0) -- (1,1) -- (4,1) -- (3,0);
\fill[gray,opacity = 0.5] (3,-3) -- (3,0) -- (3.5,0.5) -- (3.5,-2.5) -- cycle;
\fill[gray,opacity = 0.5] (3.5,-2.5) -- (3.5,-1) -- (4,-0.5) -- (4,-2) -- cycle;

\draw[thin] (0,0) -- (1,1) -- (4,1);
\draw[thin] (3,0) -- (3.5,0.5);
\draw[black] (3.5,0.5) -- (4,1) -- (4,1-1.5) --(3.5,-1) -- cycle;
\fill[gray, opacity = 0.5] (3.5,0.5) -- (4,1) -- (4,1-1.5) --(3.5,-1) -- cycle;
\draw[thin] (3,-3) -- (3.5,-2.5) -- (3.5,-1);
\draw[thin] (3.5,-2.5) -- (4,-2) -- (4,-0.5);

\filldraw[black] (4,1) circle (1.5pt);
\filldraw[black] (4,-0.5) circle (1.5pt);
\filldraw[black] (3.5,0.5) circle (1.5pt);
\filldraw[black] (3.5,-1) circle (1.5pt);

%\draw[->] (2,1.5) -- (2,0.5);
\node[anchor = south] at (2,0.2) {$d^1_3$};

\draw[thin] (3.5,1.5) -- (3,1); 
\draw[dashed,->] (3,1) -- (2.5,0.5);
\node[anchor = south] at  (3.5,1.5){$d^1_2$};

\draw[thin] (-0.5,-1) -- (0,-1);
\draw[dashed, ->] (0,-1) -- (0.5,-1);
\node[anchor = east] at  (-0.5,-1) {$d^0_1$};

%\draw[->] (-0.5,-2.5) -- (0.5,-1.5);
\node[anchor = west] at (1.5,-1.5) {$d^0_2$};

\draw[thin] (2,-3.5) -- (2,-3);
\draw[dashed,->] (2,-3) -- (2,-2.5);
\node[anchor = north] at (2,-3.5) {$d^0_3$};

%\draw[->] (4.5,0) -- (3.7,0);
\node[anchor = west] at (3.45,0) {$d^2_3$};

%\draw[->] (4.5,-1.6) -- (3.7,-1.6);
\node[anchor = west] at (3.45,-1.6) {$d^2_2$};

%\draw[->] (3.3,-3.5) -- (3.3, -1.5);
\node[anchor = north] at (3.25,-0.8) {$d^2_1$};

%Now the simplex

\begin{scope}[shift={(4,-3)}, scale = 1.3]

\draw[thin] (2,0) -- (1.23,1.17) -- (1.64,1.31) --(2,1.2) -- cycle;

\draw[thin] (1.23,1.17) -- (2.32,2.81) -- (2.81,3) -- (1.64,1.31);

\filldraw[black] (1.23,1.17) circle (1.5pt);
\filldraw[black] (2,0) circle (1.5pt);
\filldraw[black]  (1.64,1.31) circle (1.5pt);
\filldraw[black]  (2.81,3) circle (1.5pt);
\filldraw[black] (2.32,2.81) circle (1.5pt);
\filldraw[black]  (2,1.2) circle (1.5pt);

\draw[dashed] (7-1.23,1.17) -- (7-1.64,1.31) -- (7-2,1.2) -- (7-2,0);
\draw[thin] (7-2,0) -- (7-1.23,1.17);

\draw[thin] (7-1.23,1.17) -- (7-2.32,2.81) -- (7-2.81,3);

\draw[dashed] (7-2.81,3) -- (7-1.64,1.31);

\filldraw[black] (7-1.23,1.17) circle (1.5pt);
\filldraw[black] (7-2,0) circle (1.5pt);
\filldraw[black]  (7-1.64,1.31) circle (1.5pt);
\filldraw[black]  (7-2.81,3) circle (1.5pt);
\filldraw[black] (7-2.32,2.81) circle (1.5pt);
\filldraw[black]  (7-2,1.2) circle (1.5pt);

\draw[thin] (2,0) -- (5,0);
\draw[thin] (2,1.2) -- (7-2.32,2.81);
\draw[thin] (7-2.81,3) -- (2.81,3);
\draw[dashed] (2.32,2.81) --  (5,1.2);

\fill[green,opacity = 0] (2,0) -- (1.23,1.17) -- (2.32,2.81) -- (5,1.2) -- (5,0) -- cycle;

\fill[violet,opacity = 0.4] (2,0) -- (2,1.2) -- (7-2.32,2.81) -- (7-1.23,1.17) -- (5,0) -- cycle;

\fill[gray,opacity = 0.3] (2,0) -- (2,1.2) -- (1.64,1.31) -- (1.23,1.17) -- cycle;

\fill[blue,opacity = 0.3] (1.64,1.31) -- (1.23,1.17) -- (2.32,2.81) -- (2.81,3) -- cycle;

\fill[gray,opacity = 0.1] (2,1.2) -- (1.64,1.31) -- (2.81,3) -- (7 -2.81,3) -- (7-2.32,2.81) -- cycle;
\end{scope}

\node[anchor = west] at (5.9,-1.8)
{$d^0_1$};
\node[anchor = west] at (8.2,-1.8) {$d^1_2$};
\node[anchor = west] at (7,-0.5) {$d^0_3$};
\draw[->] (6,0) -- (6.6,-0.4);
\node[anchor = south east] at (6,0) {$d^0_2$};

\draw[thin] (8.5,-3.5) node[anchor = north] {$d^1_3$} -- (8.5,-3);
\draw[dashed,->] (8.5,-3) -- (8.5,-2.5);

\draw[thin] (8.5,1.5) node[anchor = south] {$d^2_1$} -- (8.5,0.9);
\draw[dashed,->] (8.5,0.9) -- (8.5,0.4);

\begin{scope}[shift = {(0.2,0)}]
\draw[->,dashed] (11.2,-1.8) -- (10.6,-1.8);
\draw (11.5,-1.8) --(11.2,-1.8);
\node[anchor = west] at (11.5,-1.8)  {$d^2_3$};
\end{scope}

\draw[->] (11.1,0.1) -- (10.5,-0.4);

\node[anchor = south west] at (11.1,0.1) {$d^2_2$};
\end{tikzpicture}
\end{center}

In general, the table below explains which cubic hyperface corresponds to which hyperplane in the truncated simplex. For the hyperface of the original simplex containing all the vertices except for $i$, the corresponding hyperplane is denoted by $D_i$.\\

\begin{center}
 \begin{tabular}{|c|c|} 
 \hline
 {\bf Cubic label} & {\bf Hyperplane in simplicial incarnation} \\ 
 \hline
$d^0_i$, $i \leq n-1$ & $Q_{i-1}$ \\ 
 \hline
$d^0_n$ & $D_n$\\
\hline 
$d^1_i$ & $D_{i-1}$\\
\hline 
 $d^2_1$ & $D_0$ \\
 \hline
 $d^2_i$, $i \geq 2$ & $P_{n-i}$ \\
 \hline
\end{tabular}
\end{center}

\begin{rem}
Cubically interpreted freehedra appear in \cite{Cha}, where a surprising connection with Dyck paths is studied.
\end{rem}

\subsection{Freehedra combinatorially}
The purely combinatorial definition of freehedra has the benefit that faces of all codimensions obtain labels. These labels are used in the main theorem of the paper.

\begin{defi} A nice $n$-expression is an expression $$s = s_l] [s_{l+1}] \ldots [s_k] | [s_0] \ldots [s_{l-1}]$$ where 
\begin{itemize}
    \item (the absence of the opening bracket for $s_l$ is not a typo)
    \item every stretch $s_i$ is a nonempty subset of $\{0,1,\ldots,n \}$
    \item for every $i$, $\max s_i = \min s_{i+1}$
    \item $|s_i| \geq 2$ if $i \neq l$ ($|s_l| = 1$ is allowed)
    \item $\min s_0 = 0$ and $\max s_k = n$
    \item in the case $l=0$ $s_0$ is placed to the left of the bar
\end{itemize}
\end{defi}

Every face of $\mathcal{F}_n$  is labelled with a nice $n$-expression. For a nice expression $s$ as above, let $L$ be the number of elements $i \in \{0,1,\ldots, n\}$ that are not present in $s$. Then the codimension of the corresponding face is $l+L$.\\

 Examples of such expression for $n = 3$ are $3]|[01][13]$ (of codimension $2+1 = 3$) or $023]|$ (of codimension $0+1 = 1$).

\begin{defi}
Consider a nice $n$-expression $s$ as above:
$$s = s_l] [s_{l+1}] \ldots [s_k] | [s_0] \ldots [s_{l-1}]$$
The {\em face transformations} that can be applied to $s$ are: 
\begin{enumerate}
    \item Drop: for some stretch $s_j$ remove some $x \in s_j$ with $\min s_j < x < \max s_j$.
    \item Inner break: replace some stretch $[s_j]$ with $[s^1_j][s^2_j]$ where $s^1_j = \{a \leq x| a \in s_j \}$ and $s^2_j = \{a \geq x| a \in s_j \}$ for some  $x \in s_j$ with $\min s_j < x < \max s_j$.
    \item Right outer break: replace the stretch $s_l]$ with $s^1_l][s^2_l]$ where $s^1_l = \{a \leq x| a \in s_l \}$ and $s^2_l = \{a \geq x| a \in s_l \}$ for some $x \in s_l$ with $x < \max s_l$.
    \item Left outer break: for $x \in s_l$ with $\min s_l < x$, replace the stretch $s_l]$ with $\{  a \geq x \ | a \in s_l \}]$, and add the stretch $[\{ a \leq x | a \in s_l \}]$ to the end of the expression after $s_{l-1}$.
\end{enumerate}
\end{defi}

For example, the expression $23]|[012]$ can be transformed into $23]|[02]$ by a drop, or into $23]|[01][12]$ by an inner break, or into $2][23]|[012]$ by a right outer break, or into $3]|[012][23]$ by a left outer break.

\begin{defi}
In $\mathcal{F}_n$, a face labelled $s'$ is a codimension 1 subface of a face labelled $s$ if $s'$ can be obtained from $s$ by one of the face transformations.
\end{defi}

The resulting abstract polytopes are precisely freehedra. The cubical notation for hyperfaces translates into into combinatorial notation for hyperfaces like this: 

\begin{itemize}
    \item $d^0_i$ corresponds to $0\ldots i-1][i-1 \ldots n]|$;
    \item $d^1_i$ corresponds to $0\ldots \widehat{i-1} \ldots n]|$, where the hat means the omission;
    \item $d^2_i$ corresponds to $i \ldots n]|[0 \ldots i]$.

\end{itemize}

Below are nice $2$-expressions and their face transformation shown on  $\mathcal{F}_2$. Drops are labelled D, inner breaks are labelled IB, left outer breaks are labelled LOB and right outer breaks are labelled ROB.

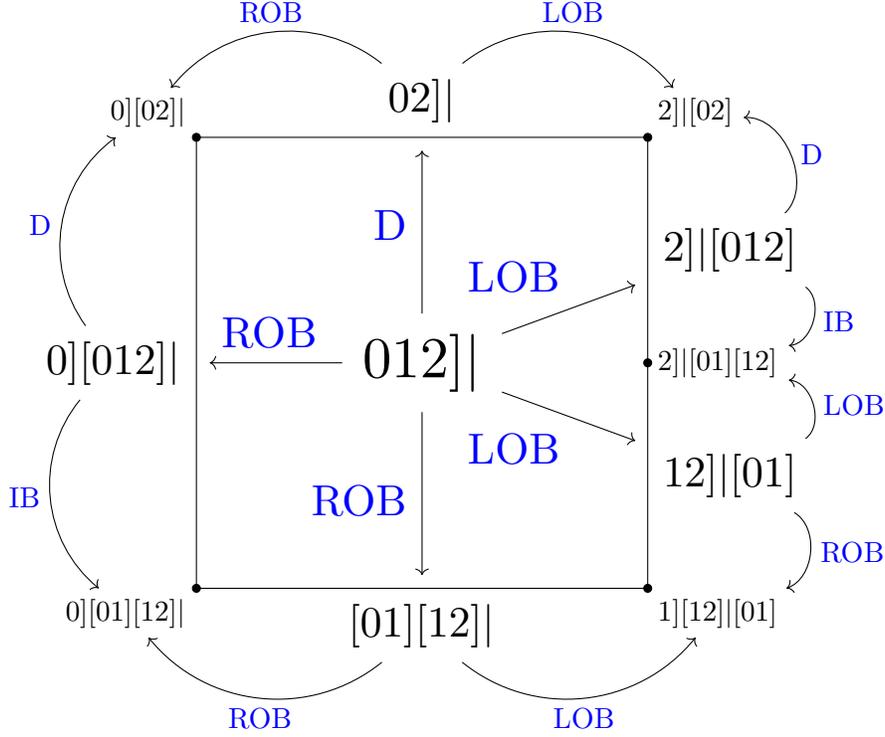
\begin{center}
\begin{tikzpicture}
\draw[thin] (0,0) rectangle (6,6);
\filldraw[black] (0,0) circle (1.5pt);
\filldraw[black] (0,6) circle (1.5pt);
\filldraw[black] (6,0) circle (1.5pt);
\filldraw[black] (6,6) circle (1.5pt);
\filldraw[black] (6,3) circle (1.5pt);

\node[scale = 2] (A) at (3,3) {$012]|$};
\node[scale = 1.5, anchor = east] (1) at (0,3) {$0][012]|$};

\node[scale = 1.5, anchor = north] (2) at (3,0) {$[01][12]|$};
\node[scale = 1.5, anchor = south] (3) at (3,6) {$02]|$};
\node[scale = 1.5, anchor = west] (4) at (6,1.5) {$12]|[01]$};
\node[scale = 1.5, anchor = west] (5) at (6,4.5) {$2]|[012]$};

\node[anchor = north east] (11) at (0,0) {$0][01][12]|$};
\node[anchor = north west] (22) at (6,0) {$1][12]|[01]$};
\node[anchor = west] (33) at (6,3) {$2]|[01][12]$};
\node[anchor = south west] (44) at (6,6) {$2]|[02]$};
\node[anchor = south east] (55) at (0,6) {$0][02]|$};

\draw[->, shorten > = 5pt] (A) -- node[above,blue,scale = 1.5]{ROB} (1);
\draw[->, shorten > = 5pt] (A) -- node[left,blue,scale = 1.5]{ROB}(2);
\draw[->, shorten > = 5pt] (A) -- node[left,blue,scale = 1.5]{D}(3);
\draw[->, shorten > = 5pt] (A) -- node[below left,blue, scale = 1.5]{LOB} (4);
\draw[->, shorten > = 5pt] (A) -- node[above left,blue, scale = 1.5]{LOB} (5);

\draw[->] (1) edge [bend right=45] node[left,blue]{IB} (11);
\draw[->] (5) edge [bend left=70] node[right,blue]{IB} (33);
\draw[->] (4) edge [bend right=70] node[right,blue]{LOB} (33);
\draw[->] (5) edge [bend right=70] node[right,blue]{D} (44);

\draw[->] (3) edge [bend left=45] node[above,blue]{LOB} (44);
\draw[->] (4) edge [bend left=65] node[right,blue]{ROB} (22);

\draw[->] (3) edge [bend right=45] node[above,blue]{ROB} (55);
\draw[->] (1) edge [bend left=45] node[left,blue]{D} (55);

\draw[->] (2) edge [bend right=45] node[below,blue]{LOB} (22);
\draw[->] (2) edge [bend left=45] node[below,blue]{ROB} (11);

\end{tikzpicture}
\end{center}

\section{Main isomorphism}
We establish an isomorphism $I$ between the set of nice expressions and the forest-tree-forest basis of $T$ from Prop \ref{isom2}. 

\begin{con}
Consider a nice $n$-expression 
$$s = s_l] [s_{l+1}] \ldots [s_k] | [s_0] \ldots [s_{l-1}]$$

We form the forest-tree-forest triple $I(s) = (F,T,G)$ as follows. Every stretch gives rise to a separate tree. The stretch $s_l$ produces $T$, the stretches $s_i$ for $i>l$ (located to the left of the bar) produce the trees of $F$ and the stretches $s_i$ for $i<l$ (located to the right of the bar) produce the trees of $G$. The trees are assembled into the triple in the following order:
$$ (F,T,G) = (\iota(s_{k}) \circ \ldots \circ \iota(s_{l+1}), \iota(s_l), \iota(s_{l-1}) \circ \ldots \circ \iota(s_{0}))$$
It remains to explain $\iota$. For a stretch $s = a_1 < \ldots <a_m$, $\iota(s)$ is a tree with $m-1$ branches, where the number of leaves on the $j$th branch is $a_{j+1}-a_j$.
\end{con}

\begin{prop}
The map $I$ above is a bijection.
\end{prop}
\begin{proof}
Having a tree-forest-tree triple $(F,T,G)$ with $n$ leaves, we form a nice $n$-expression $s = I^{-1}(F,T,G)$ as follows. Start from the rightmost branch of the rightmost tree of $G$ and move left, adding one symbol for one branch within the tree, and beginning the new stretch for the new tree. To form the next symbol of the current stretch, add the number of leaves on the current branch to previous symbol.
\end{proof}

\begin{prop}The map $I$ provides an isomorphism of chain complexes
$$ C_*(\mathcal{F}_n) \simeq  T(\aaa^n,\m;\m)$$
\end{prop}
\begin{proof}
We only need to verify that the resulting map of graded vector spaces is consistent with differentials. Consider a face of $\mathcal{F}_n$ labelled with a nice $n$-expression $s = s_l] [s_{l+1}] \ldots [s_k] | [s_0] \ldots [s_{l-1}]$, with $I(s) = (F,T,G)$. We go through the list of summands in $d(F,T,G)$ from Prop \ref{isom2}.

\begin{enumerate}
    \item The summands $U((F,T,G),B)$ for any $B$ correspond to drop transformations.
    \item The summands $S((F,T,G), B \subset F)$ correspond to inner break transformations at stretches $s_i$ for $i>l$.
    \item The summands $S((F,T,G), B \subset G)$ correspond to inner break tranformations at stretches $s_i$ for $i<l$.
    \item The summand $(F \circ T, 1, G)$ and the summands $S_l((F,T,G), B \subset T)$ correspond to left outer breaks.
    \item The summand $(F, 1, T \circ G)$ and the summands $S_r((F,T,G), B \subset T)$ correspond to right outer breaks.
\end{enumerate}
\end{proof}

Therefore we may think of forest-tree-forest triples as another collection of labels for the faces of freehedra. Recall that forests gave a collection of labels for the faces of cubes, as in Prop \ref{cubes}.

\begin{prop}
Freehedra form an CW-operadic bimodule over the CW-operad of cubes.
\end{prop}
\begin{proof}
In forest notation, the action is by forest concatenation.
\end{proof}

The theorem below summarizes the results of this section.

\begin{theo}
The underlying pair of the DG-operadic pair $(\Omega,T)$ is $C_*(I,\mathcal{F})$.
\end{theo}

\section{Projections of polyhedra}
The operadic interpretation of freehedra equips them with a natural projection from multiplihedra. We now describe it explicitly in terms of painted trees and forest-tree-forest triples. Let $T$ be a painted binary tree corresponding to a vertex of $\mathcal{J}(n)$. The projection $\pi \co \mathcal{J}(n) \to \mathcal{F}(n)$ sends $T$ to a triple $\pi(T) = (F,1,G)$, where $G$ is formed from the unpainted subtree $T'$ containing the right leaf, and $F$ is formed from $T \backslash T'$ with painting forgotten. The procedure converting these binary trees to forests is the same for $T'$ and $T \backslash T'$.

\begin{con}
 Having a binary tree, we start from the right leaf and move towards the root. Whenever we encounter a branch $B$, we create a tree with one branch having as many leaves as eventually belong to the subtree starting at $B$ (the structure of this subtree is forgotten). These trees are arranged into a forest from right to left. \\
\end{con}

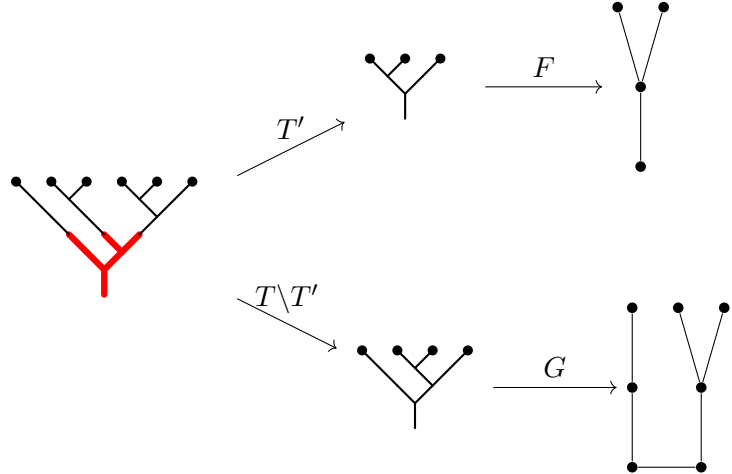
\begin{tikzpicture}
\node[scale = 2] (A) at (0,0) {\RS{B {AA LLl} {C {ALLlr} CR {Llr} {Rr}} } };
\node[scale = 2] (B) at (4,2) {\RS{I {Llr} {Rr}} };
\node[scale = 2] (C) at (4,-2) {\RS{ I {LLl} {R {L lr} {Rr} } }};
\node (E) at (7.5,-2) {\begin{forest}
for tree = {grow'=90,circle, fill, minimum width = 4pt, inner sep = 0pt, s sep = 13pt}
[{},phantom
[{},name = one [[]] ]
[{},name = two [[][]] ]
]
\draw[black] (one) -- (two);
\end{forest}};

\node (D) at (7,2) {\begin{forest}
for tree = {grow'=90,circle, fill, minimum width = 4pt, inner sep = 0pt, s sep = 13pt}
[{},phantom
[{},name = two [[][]] ]
]
\end{forest}};

\draw[->] (A) -- node[above] {$T'$} (B);
\draw[->] (A) -- node[above] {$T \backslash T'$} (C);

\draw[->] (B) -- node[above] {$F$} (D);
\draw[->] (C) -- node[above] {$G$} (E);

\end{tikzpicture}

The picture illustrates the construction of $\pi(T)$. The following proposition is now straightforward.

\begin{prop}
The projection $M_\infty^{\operatorname{Col}} \to T$ is induced by the above projection $\pi \co 
\mathcal{J} \to T$.
\end{prop}

The diagram below summarizes the projections between some families of polyhedra. Note that the projections from freehedra onto cubes and simplices are best seen at the simplicial incarnation of freehedra.

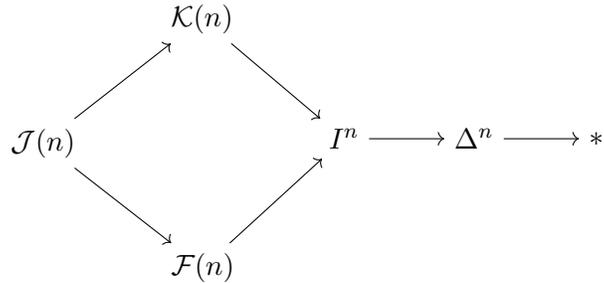
\begin{center}
 \begin{tikzpicture}

\node[] (Mult) {$\mathcal{J}(n)$};
\node[] (Assoc) [above right = of Mult] {$\mathcal{K}(n)$};
\node[] (Free) [below right = of Mult] {$\mathcal{F}(n)$};
\node[] (Cube) [below right = of Assoc] {$I^n$};
\node[] (Simp) [right = of Cube] {$\Delta^n$};
\node[] (Pt) [right = of Simp] {$*$};

%\node[] (Perm) [left = of Mult] {$\mathcal{P}(?)$};

\draw[->] (Mult) --  node[above left]{} (Assoc);
\draw[->] (Mult) --  node[below left]{} (Free);
\draw[->] (Assoc) --  node[above right]{} (Cube);
\draw[->] (Free) --  node[below right ]{} (Cube);
\draw[->] (Cube) --  node[above]{} (Simp);
\draw[->] (Simp) --  node[above]{} (Pt);

%\draw[->] (Perm) -- node[above]{1} (Mult);

\end{tikzpicture}   
\end{center}

%The family $\mathcal{P}$ means permutahedra. The Tonks projection $1 \colon \mathcal{P} \to \mathcal{J}$ is used in the celebrated paper of Saneblidze to construct a diagonal for multiplihedra (and, in composition with projection $2 \colon \mathcal{J} \to \mathcal{K}$, for associahedra). We are currently not aware of its operadic interpretation. \\

Every family of polytopes in this diagram can be interpreted operadically as the CW-counterpart of the DG-operadic bimodule in a certain $(a,m)$-colored DG-operadic pair. The partially informal table below lists these interpretations (we denote by $B$ the bimodule responsible for $A_\infty$-morphisms of DG-modules over DG-algebras).\\

\begin{center}
\begin{tabular}{|c||c|c|c|c||c|}
     \hline
     {\bf Polyhedra} & {\bf Algebras} & {\bf Modules} & {\bf \shortstack{Map of  \\ algebras}}& {\bf \shortstack{Map of  \\ modules}} & { \bf Pair}\\
     \hline
     $\mathcal{J}$ & $A_\infty$
     & $A_\infty$ & $A_\infty$ & $A_\infty$ & $(A_\infty^{\operatorname{Col}},M_\infty^{\operatorname{Col}})$ \\
     \hline 
     $\mathcal{K}$ & $A_\infty$
     & $A_\infty$ & strict & strict &  $(A_\infty^{\operatorname{Col}},A_\infty^{\operatorname{Col}})$ \\
     \hline 
     $\mathcal{F}$ & DG
     & $A_\infty$ & strict & $A_\infty$ & $(\Omega,T)$\\
     \hline 
     $I$ & DG
     & $A_\infty$ & strict & strict & $(\Omega,\Omega)$\\
     \hline 
     $\Delta$ & DG
     & DG & strict & $A_\infty$ & $(Ass^{\operatorname{Col}}, B)$ \\
     \hline
     $*$ & DG
     & DG & strict & strict & $(Ass^{\operatorname{Col}},Ass^{\operatorname{Col}})$ \\
    \hline
\end{tabular}
\end{center}

\begin{prop}
There exists the following diagram of projections between operadic pairs. Applying the functor of cellular chains to the diagram of polyhedral projections yields a part of this diagram -- namely, the bimodule part with output $\m$.

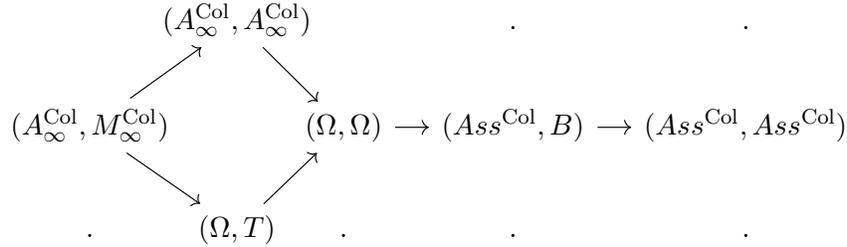
\[
\begin{tikzcd}[column sep = small]
 & [-25pt] (A_\infty^{\operatorname{Col}},A_\infty^{\operatorname{Col}})  \arrow{rd} & [-25pt] & . & . \\
(A_\infty^{\operatorname{Col}},M_\infty^{\operatorname{Col}}) \arrow{ru} \arrow{rd} & & (\Omega, \Omega) \arrow{r} & (Ass^{\operatorname{Col}},B) \arrow{r} & (Ass^{\operatorname{Col}},Ass^{\operatorname{Col}}) \\
. & (\Omega,T) \arrow{ru} & . & . & . \\
\end{tikzcd}
\]
\end{prop}

\begin{proof}
By direct inspection.
\end{proof}

\begin{rem}
The table above lists not all possible quotients of $(A_\infty^{\operatorname{Col}},M_\infty^{\operatorname{Col}})$, just the ones that are encountered in real life more frequently than never. For example, one can also consider the operadic pair controlling $A_\infty$-modules over DG-algebras, where morphisms are allowed to be $A_\infty$ both for algebras and for modules. This results in a family of where the 3-dimensional polyhedron has polygon score $(0,8,0,4)$ but does not yet appear in the Encyclopedia of Combinatorial Polytope Sequences yet. So operadic pairs can be used as a tool for obtaining new polyhedral families.
\end{rem}

\section{Hopf operadic pairs}
In the closing section we briefly discuss the diagonals for operadic pairs. The category of colored $\mathbb{N}$-sequences ${\mathbb{N}}$-$\operatorname{Seq}_{\operatorname{Col}}(\C)$ is equipped with a second tensor product, given by 
$$(\mathcal{P} \boxtimes \mathcal{Q})(c_1,\ldots,c_n;c) = \mathcal{P}(c_1,\ldots,c_n;c) \otimes \mathcal{Q}(c_1,\ldots,c_n;c)$$
This tensor product has the property that for an operad $\mathcal{P}$, $\mathcal{P} \boxtimes \mathcal{P}$ is also an operad. The definition and the proposition below are classical.
\begin{defi}
An operad $\mathcal{P}$ is called Hopf if it is equipped with a coassociative diagonal $\Delta_{\mathcal{P}} \co \mathcal{P} \to \mathcal{P} \boxtimes \mathcal{P}$. 
\end{defi}

\begin{prop}
\label{hopf}
For a Hopf operad $\mathcal{P}$, the category $\operatorname{Alg}(\mathcal{P})$ is monoidal.
\end{prop}

For any operad $\mathcal{P}$ with an operadic bimodule $\mathcal{M}$, the sequence $\mathcal{M} \boxtimes \mathcal{M}$ is an operadic bimodule over $\mathcal{P} \boxtimes \mathcal{P}$. For a Hopf operad, one can at both sides restrict along the diagonal $\Delta_{\mathcal{P}} \co \mathcal{P} \to \mathcal{P} \boxtimes \mathcal{P}$, and thus view $\mathcal{M} \boxtimes \mathcal{M}$ as a bimodule over $\mathcal{P}$ itself. This suggests the following new definition.

\begin{defi}
An operadic pair $(\mathcal{P},\mathcal{M})$ is called {\em strictly} Hopf, if $\Omega$ is a Hopf operad and there is a coassociative map of counital coalgebras $\Delta_{\mathcal{M}} \co \mathcal{M} \to \mathcal{M} \boxtimes \mathcal{M}$.
\end{defi}

\begin{prop}
For a strictly Hopf operadic pair $(\mathcal{P},\mathcal{M})$, the category $\operatorname{Alg}(\mathcal{P},\mathcal{M})$ is monoidal.
\end{prop}
\begin{proof}
The tensor product of objects follows from the Hopf structure on $\mathcal{P}$ via Prop \ref{hopf}. Consider $\mathcal{P}$-algebras $X^1 = \{X^1_c \}$, $X^2 = \{X^2_c \}$, $Y^1 = \{ Y^1_c \}$ and $Y^1 = \{ Y^1_c \}$, with morphisms $f^1 \co X^1 \to Y^1$ and $f^2: X^2 \to Y^2$, given by characteristic maps $\chi_{f^1} \co \mathcal{M} \to \underline{\operatorname{Hom}}_{X^1,Y^1}$ and $\chi_{f^2} \co \mathcal{M} \to \underline{\operatorname{Hom}}_{X^2,Y^2}$. Then the characteristic map $\chi_{f^1 \otimes f^2} \co \mathcal{M} \to \underline{\operatorname{Hom}}_{X^1 \otimes X^2,Y^1 \otimes Y^2}$ is the following composition:

\[
\begin{tikzcd}[column sep = huge]
 M \arrow[dashed]{r} \arrow{d}{\Delta_{\mathcal{M}}} & \underline{\operatorname{Hom}}_{X^1 \otimes X^2,Y^1 \otimes Y^2} \\
 M \boxtimes M \arrow{r}{\chi_{f_1} \otimes \chi_{f_2}} & \underline{\operatorname{Hom}}_{X^1,Y^1} \otimes \underline{\operatorname{Hom}}_{X^2,Y^2} \arrow{u} 
\end{tikzcd}
\]

Associativity of this tensor product follows from coassociativity of $\Delta_{\mathcal{M}}$, and consistency with compositions follows from  $\Delta_{\mathcal{M}}$ being a map of coalgebras.
\end{proof}

Unfortunately, strictly Hopf DG-operadic pairs are a rare beast, with $(\Omega, T)$ being an important non-example. $\Omega$ is indeed a Hopf operad, with a formula for $\Delta_{\Omega}$ given in Cor. 5.10 of ACD. The formula for $\Delta_T$ is given in Prop. 7.4 of ACD, but this $\Delta_T$ is neither coassociative nor a map of coalgebras. Both properties only hold up to homotopy. \\

The original constructions for $\Delta_{\Omega}$ and $\Delta_T$ are purely algebraic and involve some choices. The results of the current paper suggest that both $\Delta_{\Omega}$ and $\Delta_T$ can be interpreted as known diagonals for polyhedral families. These are obtained with the help of a partial order on faces. Assume that all cubes are embedded into $\mathbb{R}^n$ as $[0,1]^n$, and that their subdivisions into freehedra are rectangular.

\begin{defi}
For $v_1$ and $v_2$ vertices either of $I^n$ or of $\mathcal{F}(n)$, we say $v_1 \leq v_2$ if the inequality holds coordinatewise.
\end{defi}
\begin{defi}
For $F_1$ and $F_2$ faces either of $I^n$ or of $\mathcal{F}(n)$, we say $F_1 \leq F_2$ if $\max F_1 \leq \min F_2$.
\end{defi}

Then the following formula from Saneblidze defines both the cubic diagonal $\Delta \co C_*(I^n) \to C_*(I^n) \otimes C_*(I^n)$ and the freehedral diagonal $\Delta \co C_*(\mathcal{F}(n))  \to C_*(\mathcal{F}(n)) \otimes C_*(\mathcal{F}(n))$:
$$\Delta(F) = \sum_{\substack{F_1, F_2 \subset F,\text{ } F_1 \leq F_2 \\ \dim F_1 + \dim F_2 = \dim F }} F_1 \otimes F_2$$

\begin{prop}
For appropriate choices, Abad-Crainic-Dherin diagonals $\Delta_\Omega$ and $\Delta_T$ coincide with Saneblidze diagonals given by the formula above.
\end{prop}

The proof requires translating the original constructions of $\Delta_\Omega$ and $\Delta_T$ to operadic language, which is technically involved. Thus we delay the proof until the follow up paper, where we define {\em weakly} Hopf operadic pairs and upgrade $( \Delta_\Omega, \Delta_T)$ to weakly Hopf structure.


\begin{thebibliography}{}


\bibitem[AC]{AC} C.A. Abad, M. Crainic. ``Representations up to homotopy and Bott’s spectral sequence for Lie groupoids". Advances in Mathematics, 248, 416-452 (2013).



\bibitem[ACD]{ACD} C.A. Abad, M. Crainic, B. Dherin. ``Tensor products of representations up to homotopy". Journal of Homotopy and Related Structures, 6(2), 239–288 (2011).

\bibitem[Cha]{Cha} F. Chapoton. ``Some properties of a new partial order on Dyck paths". Algebraic Combinatorics, 3(2), 433-463 (2020).


\bibitem[For]{For} S. Forcey.  ``Convex Hull Realizations of the Multiplihedra". Topology and its Applications, 56(2), 326-347 (2008).

\bibitem[Kel]{Kel} B. Keller. ``Introduction to A-infinity algebras and modules". Homology, Homotopy and Applications, 3(1), 1-35 (2001).
 

%\bibitem[Lod]{Lod} J.-L. Loday. ``Realization of the Stasheff polytope'', Archiv der Mathematik, 83, 267–278 (2004).

\bibitem[RS]{RS} M. Rivera, S. Saneblidze.  ``A combinatorial model for the free loop fibration". Bulletin of the London Mathematical Society, 50(6), 1085-1101 (2018).

\bibitem[San]{San} S. Saneblidze.  ``The bitwisted Cartesian model for the free loop fibration". Topology and its Applications, 156(5), 897-910 (2009).


\bibitem[Sta]{Sta} J. Stasheff. ``Homotopy associativity of H-spaces. I, II". Transactions of the American Mathematical Society, 108, 293–312 (1963).




\bibitem[Tam]{Tam} D. Tamari. ``Monoïdes préordonnés et chaînes de Malcev", Thèse, Université de Paris (1951).



\end{thebibliography}
\end{document}